\documentclass[12pt]{amsart}

\usepackage{amsmath,amssymb,amsfonts}
\usepackage{mathrsfs}
\usepackage{mathtools}
\usepackage{enumitem}
\usepackage{xcolor}
\usepackage{tikz-cd}
\usepackage{hyperref}
\usepackage[alphabetic,msc-links]{amsrefs}
\hypersetup{
  colorlinks=false,
  pdfborder={0 0 0.7},
  pdfborderstyle={/S/D/D[3 2]/W 0.7},
  linkbordercolor={0.10 0.52 0.28},
  citebordercolor={0.04 0.45 0.22},
  urlbordercolor={0.00 0.42 0.30},
  pdftitle={McKean Rigidity for Cocompact Negatively Curved Manifolds and the p-Laplacian},
  pdfauthor={Kuntao Jin, Shing-Tung Yau, and Bo Zhu},
  pdfsubject={Spectral rigidity under an upper sectional-curvature bound},
  pdfkeywords={McKean inequality, p-Laplacian, bottom spectrum, Busemann function, leafwise diffusion}
}

\usepackage{geometry}
\newtheorem{theorem}{Theorem}[section]

\newtheorem{proposition}[theorem]{Proposition}
\newtheorem{lemma}[theorem]{Lemma}

\newtheorem{problem}[theorem]{Problem}

\newtheorem*{acknowledgements}{Acknowledgements}

\newtheorem*{hilleyosida}{Hille--Yosida theorem (contraction form)}
\theoremstyle{definition}
\newtheorem{definition}[theorem]{Definition}
\newtheorem{example}[theorem]{Example}
\newtheorem{remark}[theorem]{Remark}
\newtheorem*{unnumberedremark}{Remark}

\numberwithin{equation}{section}

\newcommand{\ee}{{\mathrm{e}}}

\newcommand{\mc}{\mathcal}
\newcommand{\ms}{\mathscr}

\newcommand{\wti}{\widetilde}
\newcommand{\mr}{\mathrm}

\newcommand{\R}{\mathbb{R}}
\newcommand{\bH}{\mathbb{H}}
\newcommand{\supp}{\mathrm{supp}}
\newcommand{\vol}{\operatorname{vol}}
\newcommand{\area}{\operatorname{area}}
\newcommand{\tr}{\mathrm{tr}}
\DeclareMathOperator{\Ric}{Ric}

\begin{document}

\title[McKean rigidity]{McKean Rigidity for Cocompact Negatively Curved
Manifolds and the \(p\)-Laplacian}
\date{\today}

\author{Kuntao Jin}
\address[Kuntao Jin]{Department of Mathematical Sciences, Tsinghua University}
\email{jkt25@mails.tsinghua.edu.cn}

\author{Bo Zhu}
\address[Bo Zhu]{Yau Mathematical Sciences Center,
Tsinghua University}
\email{zhub@tsinghua.edu.cn}
\thanks{Bo Zhu is supported by NSFC~12501066.}


\begin{abstract}
Let \((M^m,g)\) be a closed Riemannian manifold with
\(\sec_g\leq-1\).  We prove that the bottom spectrum of its universal
cover attains McKean's lower bound if and only if the universal cover is
hyperbolic space of constant sectional curvature \(-1\).  More generally,
for every \(1<p<\infty\), the variational \(p\)-fundamental tone satisfies
\[
  \lambda_{1,p}(\wti M)
  \geq\left(\frac{m-1}{p}\right)^p,
\]
and equality for some \(p\in(1,\infty)\) holds if and only if
\(\wti M\cong\bH^m(-1)\).  In that case, equality holds for every
\(p\in(1,\infty)\).  The proof converts the two McKean defects of a
minimizing sequence into a stationary probability measure on the compact
horospherical suspension; heat-kernel positivity then forces its
zero-defect support to contain a complete leaf. 
\end{abstract}
\maketitle
\section{Introduction}
Let \((X^m,g_X)\) be a complete Riemannian manifold, where \(m\geq2\).  Its
bottom spectrum is
\begin{equation}\label{eq:bottom-spectrum}
  \lambda_1(X)
  =\inf_{0\neq u\in C_c^\infty(X)}
  \frac{\int_X|\nabla u|^2\,dV}
       {\int_Xu^2\,dV}.
\end{equation}
If \(X\) is simply connected and \(\sec_{g_X}\leq-1\), McKean proved the
following sharp estimate (see \cite{McKean1970}, pp.~360--365):
\begin{equation}\label{eq:mckean-bound}
  \lambda_1(X)
  \geq\frac{(m-1)^2}{4}
  =\lambda_1\bigl(\bH^m(-1)\bigr).
\end{equation}
For a general complete, simply connected manifold, however, equality in
McKean's estimate is not rigid: the bottom spectrum alone does
not determine the geometry of \(X\).  The examples in
Subsection~\ref{subsec:role-cocompactness} explain the two points at which
cocompactness enters the rigidity argument.

\medskip
We prove rigidity in the cocompact setting: \(X=\wti M\) is the universal
Riemannian cover of a closed manifold.

\begin{theorem}\label{thm:main}
Let \((M^m,g)\) be a closed connected Riemannian manifold of dimension
\(m\geq2\) such that
\(
  \sec_g\leq-1.
\)
Then
\[
  \lambda_1(\wti M)\geq\frac{(m-1)^2}{4}.
\]
Equality holds if and only if
\[
  (\wti M,\wti g)\cong\bH^m(-1).
\]
\end{theorem}

\medskip
The same rigidity holds for the variational \(p\)-Laplacian.  For a
complete Riemannian manifold \(X\) and \(1<p<\infty\), set
\[
  \lambda_{1,p}(X)
  :=
  \inf_{0\neq u\in C_c^\infty(X)}
  \frac{\|\nabla u\|_{L^p(X)}^p}
       {\|u\|_{L^p(X)}^p}.
\]
The sharp lower bound under \(\sec\leq-1\) is known
(see \cite{Poliquin2014}, Proposition~4.5, and
\cite{CarvalhoCavalcante2022}, Corollary~1.1).  Our concern is rigidity
in the equality case for universal covers of closed manifolds.

\begin{theorem}
\label{thm:p-rigidity}
Let \((M^m,g)\) be a closed connected Riemannian manifold of dimension
\(m\geq2\) satisfying \(\sec_g\leq-1\).  Then, for every
\(1<p<\infty\),
\[
  \lambda_{1,p}(\wti M)
  \geq\left(\frac{m-1}{p}\right)^p.
\]
Equality holds for some \(p\in(1,\infty)\) if and only if
\[
  (\wti M,\wti g)\cong\bH^m(-1).
\]
In this case, equality holds for every \(p\in(1,\infty)\).
\end{theorem}

To the best of our knowledge, neither the quadratic equality-rigidity
statement under the upper curvature bound \(\sec_g\leq-1\) nor its
\(p\)-Laplacian extension was previously known.  In this paper, we first
treat the quadratic case \(p=2\), where the square identity separates the
two McKean defects.  For general \(p\), uniform convexity of
\(Y\mapsto|Y|^p\) yields, for a normalized \(p\)-minimizing sequence
\(u_j\) and \(\rho_j=u_j^p\),
\[
  \left\|
    \nabla\rho_j+(m-1)\rho_j\nabla B
  \right\|_{L^1}
  \longrightarrow0.
\]
The coefficient \(m-1\) is independent of \(p\); hence the limiting measure
is stationary for the same leafwise diffusion as in the quadratic case,
and the geometric and dynamical parts of the argument are unchanged.  The
nonlinear defect calculation and the proof of
Theorem~\ref{thm:p-rigidity} are deferred to
Appendix~\ref{app:p-rigidity}.

\bigskip
\subsection*{Related rigidity results under curvature bounds}

For a complete Riemannian manifold \((X^m,g)\) with
\(\Ric_g\geq-(m-1)g\), Shiu-Yuen Cheng proved the sharp estimate
(see \cite{Cheng1975}):
\(\lambda_1(X)\leq (m-1)^2/4\). When $m\geq 3$,
Peter Li and Jiaping Wang analyzed the equality
case on general complete manifolds by harmonic-function methods; their
splitting theorems identify explicit warped products as the only equality
cases with more than one end (see \cite{LiWang2001} and
\cite{LiWang2002}).  Thus equality alone does not imply hyperbolicity in
the noncocompact setting.  For universal covers of closed manifolds,
however, Xiaodong Wang proved full rigidity using Martin-boundary harmonic
functions and Kaimanovich entropy: if \((M^m,g)\) is closed,
\(\Ric_g\geq-(m-1)g\), and
\(\lambda_1(\wti M)=(m-1)^2/4\), then
\(\wti M\cong\bH^m(-1)\)
(see \cite{Wang2008}, Theorem~4).

Scalar-curvature analogues follow the same pattern.  In dimension three,
Munteanu and Jiaping Wang used harmonic-function theory to prove the sharp
estimate, normalized to the hyperbolic curvature scale \(-1\):
\(\operatorname{Sc}_X\geq-6\) implies \(\lambda_1(X)\leq 1\)
for a complete \(X^3\) with finitely many ends and finite first Betti
number (see \cite{MunteanuWang2024}, Theorem~1.1, and
\cite{MunteanuWang2026}, Theorem~1.1).\footnote{Munteanu and Wang assume a
Ricci lower bound in the first paper and remove this assumption in the
second (see \cite{MunteanuWang2024} and \cite{MunteanuWang2026}).} In
higher dimensions, Jinmin Wang and Bo Zhu proved the corresponding
sharp estimate, together with rigidity in the equality case, for universal
covers of closed manifolds.  If
\(M\) is rationally essential, \(\wti M\) is spin, and \(\pi_1(M)\)
satisfies the Strong Novikov Conjecture, then
\(\operatorname{Sc}_g\geq-m(m-1)\) implies
\(\lambda_1(\wti M)\leq(m-1)^2/4\), and equality forces
\((M,g)\) to be hyperbolic
(see \cite{WangZhu2024}, Theorem~1.2, and
\cite{WangZhu2026}, Theorem~1.2).\footnote{Wang and Zhu prove the sharp
estimate in the first paper and rigidity in the second
(see \cite{WangZhu2024} and \cite{WangZhu2026}).  Their rigidity argument
uses almost-harmonic spinors to
recover the Einstein identity
\(\Ric_g=-(m-1)g\), after which Xiaodong Wang's rigidity theorem applies.}

Related rigidity results begin with stronger horospherical assumptions.
Besson, Courtois, and Hersonsky obtain hyperbolicity from geometric
conditions on horospheres, while Zimmer proves rigidity under asymptotic
harmonicity, which requires every Busemann function to have the same
constant Laplacian (see
\cites{BessonCourtoisHersonsky2025Flat,
BessonCourtoisHersonsky2025Horospherical,Zimmer2012}).  By contrast,
spectral equality here initially yields only averaged defect limits.  The
stationary-measure argument produces a single
\(\xi_*\in\partial_\infty\wti M\) satisfying
\(\Delta B_{\xi_*}=m-1\), and under the present curvature bound this
identity suffices to prove hyperbolicity.

\subsection*{The invariant perspective: isoperimetry and entropy}

The preceding results formulate rigidity through curvature inequalities.
A complementary viewpoint uses two global geometric invariants: the
Cheeger constant and the volume entropy.
The relation between isoperimetric constants and the first eigenvalue was
further developed by Yau in the compact setting
(see \cite{Yau1975Isoperimetric}).

For a complete \(X\), and for a universal cover \(\wti M\) of a closed
manifold, set
\[
\begin{aligned}
  h_{\mathrm{iso}}(X)
  &:=
  \inf_{\Omega\Subset X}
  \frac{\area(\partial\Omega)}{\vol(\Omega)},\\
  h_{\mathrm{vol}}(\wti M)
  &:=
  \lim_{R\to\infty}\frac{1}{R}\log\vol B(o,R).
\end{aligned}
\]
Then Cheeger's inequality and Brooks's volume-growth
estimate give (see \cite{Cheeger1970} and \cite{Brooks1981})
\begin{equation}\label{eq:cheeger-spectrum-entropy-chain}
  h_{\mathrm{iso}}(\wti M)^2
  \leq 4\lambda_1(\wti M)
  \leq h_{\mathrm{vol}}(\wti M)^2.
\end{equation}

Under a Ricci lower bound \(\Ric_g\geq-(m-1)g\), Bishop--Gromov comparison
gives \(h_{\mathrm{vol}}(\wti M)\leq m-1\).  Hence equality in the sharp
spectral bound forces \(h_{\mathrm{vol}}(\wti M)=m-1\), and the entropy
rigidity theorem of Ledrappier--X.~Wang, with Gang Liu's short proof, yields
\(\wti M\cong\bH^m(-1)\)
(see \cite{LedrappierWang2010}, Theorem~2, and
\cite{Liu2011}, Theorem~1).

Under the sectional-curvature assumption of this paper, G\"unther's volume
comparison theorem gives the opposite inequality (see \cite{Gunther1960}):
\[
  \sec_g\leq-1
  \quad\Longrightarrow\quad
  h_{\mathrm{vol}}(\wti M)\geq m-1.
\]
Together with the right-hand inequality in
\eqref{eq:cheeger-spectrum-entropy-chain}, spectral equality again gives
only \(h_{\mathrm{vol}}(\wti M)\geq m-1\), which is precisely the
comparison estimate above.  Thus the spectral hypothesis supplies no
reverse volume-growth inequality, and the entropy route does not close
the argument.

Cheeger's inequality yields a necessary consequence of spectral equality,
but it does not by itself imply rigidity.  Indeed, for every relatively
compact smooth domain \(\Omega\subset\wti M\), Busemann comparison and the
divergence theorem give
\[
  (m-1)\vol(\Omega)
  \leq\int_\Omega\Delta B_\xi\,dV
  =\int_{\partial\Omega}\langle\nabla B_\xi,\nu\rangle\,dA
  \leq\area(\partial\Omega).
\]
Taking the infimum over \(\Omega\) yields
\(h_{\mathrm{iso}}(\wti M)\geq m-1\).  If spectral equality holds, Cheeger's
inequality gives \(h_{\mathrm{iso}}(\wti M)\leq m-1\), and therefore
\[
  h_{\mathrm{iso}}(\wti M)=m-1.
\]
For \(p=1\), define
\[
  \lambda_{1,1}(X)
  :=
  \inf_{0\neq u\in C_c^\infty(X)}
  \frac{\int_X|\nabla u|\,dV}
       {\int_X|u|\,dV}.
\]
The coarea formula and smooth approximation of characteristic functions
give
\[
  \lambda_{1,1}(X)=h_{\mathrm{iso}}(X).
\]

The proof of Theorem~\ref{thm:main} uses spectral minimizing sequences
directly.  Entropy equality is not assumed; it follows only after
hyperbolicity has been proved.

\subsection*{Outline of the proof}

Assume that \(\lambda_1(\wti M)=(m-1)^2/4\).  To prove
Theorem~\ref{thm:main}, it is enough to find
\(\xi_*\in\partial_\infty\wti M\) such that \(B=B_{\xi_*}\) satisfies
\begin{equation}\label{eq:busemann-conclusion}
  |\nabla B|=1,
  \quad
  \Delta B=m-1
\end{equation}
throughout \(\wti M\).  The unit-gradient identity is automatic for a
Busemann function; the task is therefore to obtain the constant-Laplacian
identity.  Busemann Hessian comparison then converts
\eqref{eq:busemann-conclusion} into hyperbolic rigidity.

The central point is that the two McKean defects control complementary
properties of a single limiting measure on the compact horospherical
suspension.  The first defect gives stationarity for a leafwise diffusion
with drift coefficient \(m-1\), while the second confines the support to
the zero set of the Busemann Laplacian defect.  Positivity of the leafwise
diffusion makes the stationary support leaf-saturated, thereby propagating
the Busemann equality along a complete leaf.  This stationary-measure
viewpoint is inspired by Ledrappier--X.~Wang
(see \cite{LedrappierWang2010}, Theorems~1 and~2).  The required diffusion
theory follows Candel's Hille--Yosida construction
(see \cite{Candel2003}, Theorem~4.14 and
Propositions~4.15--4.16); its detailed verification is deferred to
Section~\ref{sec:candel-semigroup-construction}, after the main rigidity
argument.

\begin{enumerate}[label=\textup{(\arabic*)}]
\item Section~\ref{sec:busemann-defects} derives the McKean square
identity and produces a normalized minimizing sequence \(u_j\) for which
both defects vanish; see \eqref{eq:first-defect} and
\eqref{eq:second-defect}.

\item Section~\ref{sec:suspension} introduces the compact horospherical
suspension
\[
  Z=(\wti M\times\partial_\infty\wti M)/\Gamma,
  \quad Z\cong SM.
\]
Its leaves are covered by \(\wti M\), and its compactness prevents the
quadratic mass of \(u_j\) from being lost at infinity.

\item In Section~\ref{sec:limiting-measure}, we push \(u_j^2dV\) to \(Z\) along
the leaf determined by \(\xi_0\).  A subsequence converges to a probability
measure \(\mu\).  For
\(D([x,\xi])=\Delta B_\xi(x)-(m-1)\), the second defect gives
\[
  \supp\mu\subseteq D^{-1}(0).
\]

\item In Section~\ref{sec:stationarity}, the first defect gives
\(\int_ZLf\,d\mu=0\) for the leafwise drift operator
\[
  L=\Delta^{\mathrm{leaf}}
    -(m-1)\langle V,\nabla^{\mathrm{leaf}}\,\cdot\,\rangle.
\]
The associated diffusion makes \(\mu\) stationary, and its leafwise
positivity makes \(\supp\mu\) leaf-saturated.  Since
\(\supp\mu\subseteq D^{-1}(0)\), one complete leaf satisfies \(D=0\);
hence \(\Delta B_{\xi_*}=m-1\) on \(\wti M\).

\item In Section~\ref{sec:rigidity}, equality in Busemann Hessian comparison
gives
\[
  \nabla^2B=\wti g-dB\otimes dB,
\]
so the gradient flow yields an exponential warped product.  The curvature
bound inherited from the compact quotient forces its cross-section to be
flat, and therefore \((\wti M,\wti g)\cong\bH^m(-1)\).

\item After the main proof, Section~\ref{sec:candel-semigroup-construction}
develops the leafwise semigroup used in
Section~\ref{sec:stationarity}.  It gives the construction at the finite
transverse regularity needed here, proves gradient control on the
classical domain and the core property of that domain, and establishes
leaf saturation of stationary supports.
\end{enumerate}

Appendix~\ref{app:p-rigidity} proves
Theorem~\ref{thm:p-rigidity}.  The uniform-convexity defect for the
\(p\)-energy yields an \(L^1\) transport equation for the probability
density \(u_j^p\).  Since its drift coefficient is again \(m-1\), the
suspension, stationary-measure, and support-saturation arguments from the
quadratic case apply without change.

\begin{acknowledgements}
The authors are grateful to Jiaping Wang for bringing this problem to their attention and to Shing-Tung Yau and Xiaodong Wang for their valuable comments and suggestions. The first author also thanks his advisor, Guoyi Xu, for his guidance and mentorship.
\end{acknowledgements}

\section{Busemann functions and the McKean defects}
\label{sec:busemann-defects}

No cocompact group action is used in this section.  Throughout, let
\((X^m,g_X)\) be a complete simply connected Riemannian manifold, where
\(m\geq2\), and assume that, for some \(A\geq1\),
\begin{equation}\label{eq:pinched-curvature}
  -A^2\leq\sec_{g_X}\leq-1.
\end{equation}
Then \(X\) is a pinched negatively curved Hadamard manifold.  The
facts about the visual boundary and Busemann functions used in this
section are classical
(see \cite{EberleinONeill1973}, Sections~1--3, and
\cite{HeintzeImHof1977}, Sections~2--3).  The excess in McKean's
inequality will be written as the sum of a first-order defect and a
Busemann Laplacian defect.  Spectral equality forces both terms to vanish
along a single normalized minimizing sequence.

\subsection{The visual boundary and Busemann comparison}

A \emph{geodesic ray} is a unit-speed geodesic
\(c\colon[0,\infty)\to X\); thus
\[
  d\bigl(c(s),c(t)\bigr)=|s-t|
  \quad (s,t\geq0).
\]
Two geodesic rays \(c_1\) and \(c_2\) are called \emph{asymptotic}, written
\(c_1\sim c_2\), if
\[
  \sup_{t\geq0}d\bigl(c_1(t),c_2(t)\bigr)<\infty.
\]
For unit-speed rays, this is equivalent to requiring their images to have
finite Hausdorff distance.  The triangle inequality shows that \(\sim\) is
an equivalence relation.  The \emph{visual boundary} of \(X\) is
\[
  \partial_\infty X
  :=
  \{\text{geodesic rays in \(X\)}\}/\mathord{\sim}.
\]
The class \([c]\), also denoted by \(c(+\infty)\), is the endpoint of \(c\)
at infinity.

We next describe this boundary from a fixed base point \(o\in X\).  Set
\[
  S_oX:=\{v\in T_oX:|v|_{g_X}=1\},
  \quad
  c_v(t):=\exp_o(tv).
\]
Every \(\xi\in\partial_\infty X\) has a unique representative
\(c_{o,\xi}\) issuing from \(o\).  Existence follows by taking a locally
uniform limit of the minimizing segments from \(o\) to points tending to
\(\xi\).  For uniqueness, if \(\gamma_1,\gamma_2\) are asymptotic rays
issuing from \(o\), then
\[
  f(t):=d\bigl(\gamma_1(t),\gamma_2(t)\bigr)
\]
is bounded and convex with \(f(0)=0\).  Thus
\(f(s)\leq (s/t)f(t)\) for \(0<s<t\); letting \(t\to\infty\) gives
\(\gamma_1=\gamma_2\).

It follows that the endpoint map
\[
  \Phi_o\colon S_oX\longrightarrow\partial_\infty X,
  \quad
  \Phi_o(v):=[c_v],
\]
is a bijection.  The \emph{cone topology} on \(\partial_\infty X\) is the
topology transported from \(S_oX\) by \(\Phi_o\).  Thus
\(\xi_j\to\xi\) precisely when
\[
  \dot c_{o,\xi_j}(0)\longrightarrow\dot c_{o,\xi}(0)
  \quad\text{in }S_oX.
\]
Equivalently, \(c_{o,\xi_j}\to c_{o,\xi}\) uniformly on every bounded time
interval.

For another base point \(o'\in X\), define the reanchoring map by
\[
  R_{o,o'}
  :=
  \Phi_{o'}^{-1}\circ\Phi_o
  \colon S_oX\longrightarrow S_{o'}X.
\]
The same compactness and convexity argument shows that \(R_{o,o'}\) and
\(R_{o',o}\) are continuous.  Hence \(R_{o,o'}\) is a homeomorphism, so the
cone topology is independent of the base point.

Finally, an isometry \(F\in\operatorname{Isom}(X)\) acts on the visual
boundary by
\[
  F_\infty([c]):=[F\circ c].
\]
This is well defined because \(F\) preserves distances.  In the endpoint
coordinates it has the form
\[
  F_\infty
  =
  \Phi_{F(o)}
  \circ\left.dF_o\right|_{S_oX}
  \circ\Phi_o^{-1},
\]
so \(F_\infty\) is a homeomorphism of \(\partial_\infty X\).

For \(\xi\in\partial_\infty X\), define the Busemann function normalized at
\(o\) by
\begin{equation}\label{eq:busemann-definition}
  B_{\xi,o}(x)
  :=\lim_{t\to\infty}
  \bigl(d(x,c_{o,\xi}(t))-t\bigr).
\end{equation}
To justify the definition, set
\[
  h_t(x):=d(x,c_{o,\xi}(t))-t.
\]
The triangle inequality shows that \(t\mapsto h_t(x)\) is nonincreasing and
that
\[
  -d(x,o)\leq h_t(x)\leq d(x,o).
\]
Hence the limit in \eqref{eq:busemann-definition} exists, is finite, and
satisfies \(B_{\xi,o}(o)=0\).  Since the functions \(h_t\) are uniformly
\(1\)-Lipschitz, they converge uniformly on compact subsets, and
\(B_{\xi,o}\) is \(1\)-Lipschitz.

Changing the base point changes only the additive normalization.  More
precisely, for \(o'\in X\),
\[
  B_{\xi,o'}(x)
  =
  B_{\xi,o}(x)-B_{\xi,o}(o').
\]
It is therefore natural to introduce the Busemann cocycle
\[
  \beta_\xi(x,y)
  :=
  B_{\xi,o}(x)-B_{\xi,o}(y);
\]
by the base-point change identity, this definition is independent of \(o\).
It satisfies
\[
  \beta_\xi(x,z)
  =
  \beta_\xi(x,y)+\beta_\xi(y,z),
  \quad
  |\beta_\xi(x,y)|\leq d(x,y).
\]
Moreover, if \(F\in\operatorname{Isom}(X)\), then
\[
  B_{F_\infty(\xi),F(o)}(F(x))=B_{\xi,o}(x),
  \quad
  \beta_{F_\infty(\xi)}(F(x),F(y))=\beta_\xi(x,y).
\]
Since the gradient, Hessian, and Laplacian are unchanged by adding a
constant, we suppress the base point and write \(B_\xi\) whenever only
derivatives are involved.  The curvature pinching in
\eqref{eq:pinched-curvature} yields the differential comparison estimates
used below.

\begin{proposition}
\label{prop:busemann-comparison}
Let \((X^m,g_X)\) be a complete simply connected Riemannian manifold,
where \(m\geq2\), and assume that
\[
  -A^2\leq\sec_{g_X}\leq-1
\]
for some \(A\geq1\).
Then 
\begin{enumerate}[label=\textup{(\roman*)}]
\item For every \(\xi\in\partial_\infty X\), the Busemann function
  \(B_\xi\) is of class \(C^2\).  If \(v_{x,\xi}\in S_xX\) is the initial
  velocity of the unique ray from \(x\) to \(\xi\), then
  \begin{equation}\label{eq:unit-busemann-gradient}
    \nabla B_\xi(x)=-v_{x,\xi},
    \quad
    |\nabla B_\xi(x)|=1.
  \end{equation}

\item For every \(\xi\in\partial_\infty X\),
  \addtocounter{equation}{1}
  \begin{equation}\label{eq:busemann-laplacian-comparison}
    \begin{aligned}
      g_X-dB_\xi\otimes dB_\xi
      &\leq\nabla^2B_\xi
      \leq A(g_X-dB_\xi\otimes dB_\xi),\\
      m-1&\leq\Delta B_\xi\leq A(m-1).
    \end{aligned}
  \end{equation}

\item Let \(\operatorname{pr}_X\colon
  X\times\partial_\infty X\to X\) be the projection.  Then
  \[
    (x,\xi)\longmapsto\nabla B_\xi(x),
    \quad
    (x,\xi)\longmapsto\nabla^2B_\xi(x)
  \]
  define continuous sections of
  \(\operatorname{pr}_X^*TX\) and
  \(\operatorname{pr}_X^*\operatorname{Sym}^2(T^*X)\), respectively, and
  \[
    (x,\xi)\longmapsto\Delta B_\xi(x)
  \]
  is a continuous function on \(X\times\partial_\infty X\).
\end{enumerate}
\end{proposition}

\begin{proof}
Fix \(\xi\in\partial_\infty X\) and \(x\in X\), and write
\(c=c_{x,\xi}\).  For \(R>0\), set
\[
  b_R(y):=d(y,c(R))-R.
\]
On every compact subset of \(X\), the functions \(b_R\) are smooth for
all sufficiently large \(R\) and converge to the Busemann function
normalized by \(B_\xi(x)=0\).

We use the standard radial-field convergence theorem.  The radial fields
\(-\nabla b_R\) converge locally uniformly to
\(v_{\,\cdot,\xi}\), and their covariant derivatives converge locally
uniformly to the derivatives determined by the stable Jacobi fields along
the rays to \(\xi\).  Equivalently, \(b_R\to B_\xi\) in
\(C^2\) on compact subsets.  This gives
\[
  B_\xi\in C^2(X),
  \quad
  \nabla B_\xi(x)=-v_{x,\xi},
  \quad
  |\nabla B_\xi|=1,
\]
and proves \textup{(i)}.  The local \(C^2\)-convergence, including its
uniformity on compact families, is the content of the stable-Jacobi-field
argument cited after the proof.

For \textup{(ii)}, decompose
\[
  w=a\,v_{x,\xi}+w^\perp,
  \quad
  w^\perp\perp v_{x,\xi}.
\]
Hessian comparison for the distance from \(c(R)\) gives
\[
  \coth R\,|w^\perp|^2
  \leq \nabla^2b_R(x)(w,w)
  \leq A\coth(AR)\,|w^\perp|^2.
\]
Passing to the \(C^2\)-limit and using
\(\nabla B_\xi=-v_{x,\xi}\), we obtain
\[
  |w|^2-\langle w,\nabla B_\xi\rangle^2
  \leq\nabla^2B_\xi(w,w)
  \leq
  A\bigl(|w|^2-\langle w,\nabla B_\xi\rangle^2\bigr).
\]
This is the tensor inequality in
\eqref{eq:busemann-laplacian-comparison}.  Taking the trace in an
orthonormal basis whose first vector is \(v_{x,\xi}\) gives the two
Laplacian bounds.

It remains to prove the joint continuity in \textup{(iii)}.  The map
\[
  (x,\xi)\longmapsto v_{x,\xi}
\]
is continuous by the cone topology, so the gradient formula gives the
joint continuity of \(\nabla B_\xi(x)\).

For the Hessian, fix \(T>0\) and define
\[
  H^T_{x,\xi}
  :=
  \left.
  \nabla_y^2\bigl(d(y,c_{x,\xi}(T))-T\bigr)
  \right|_{y=x}.
\]
Because a Hadamard manifold has no cut locus and
\((x,\xi)\mapsto c_{x,\xi}(T)\) is continuous, \(H^T\) is a continuous
section of
\(\operatorname{pr}_X^*\operatorname{Sym}^2(T^*X)\).

Let \(J_w^T\) be the normal Jacobi field along
\(c_{x,\xi}|_{[0,T]}\) with \(J_w^T(0)=w\) and \(J_w^T(T)=0\), and let
\(J_w^s\) be the bounded normal Jacobi field with \(J_w^s(0)=w\).
The Euclidean comparison estimate in the construction of \(J_w^s\)
gives
\[
  \bigl|(J_w^T)'(0)-(J_w^s)'(0)\bigr|
  \leq\frac{|w|}{T}.
\]
The Jacobi-field formulas for the two Hessians, together with their
vanishing in the \(v_{x,\xi}\)-direction, therefore yield
\[
  \bigl\|H^T_{x,\xi}-\nabla^2B_\xi(x)\bigr\|
  \leq\frac1T.
\]
The estimate is independent of \((x,\xi)\).  Hence
\(\nabla^2B_\xi(x)\) is the uniform limit of the continuous sections
\(H^T\) and is jointly continuous.  Taking the trace proves the joint
continuity of \(\Delta B_\xi(x)\).
\end{proof}

Proposition~\ref{prop:busemann-comparison} establishes the regularity and
comparison facts needed below.  The \(C^2\)-regularity in part~\textup{(i)}
follows from the stable-Jacobi-field argument (see
\cite{HeintzeImHof1977}, Proposition~3.1 and Lemma~2.2).  The
joint-continuity statement in part~\textup{(iii)} is also standard
(see \cite{EberleinONeill1973}, Proposition~3.5, and
\cite{HeintzeImHof1977}, Proposition~3.2).  The two-sided curvature bound
supplies the regularity used here; Busemann functions on an arbitrary
Hadamard manifold need not be \(C^2\).  In the model case of constant
sectional curvature \(-1\), the comparison estimate is sharp:
\[
  \nabla^2B_\xi=g_X-dB_\xi\otimes dB_\xi,
  \quad
  \Delta B_\xi=m-1.
\]
Thus, under \eqref{eq:pinched-curvature}, both \(\nabla^2B_\xi\) and
\(\Delta B_\xi\) are classical pointwise quantities, and the square
identity below follows by ordinary integration by parts.

\subsection{The McKean square identity}
The following identity is essential for the arguments that follow, and it is independent of curvature.
\begin{lemma}
\label{lem:mckean-square}
Let \((X^m,g_X)\) be a Riemannian manifold without boundary, where
\(m\geq2\), and let \(B\in C^2(X)\) satisfy \(|\nabla B|=1\).  Then, for every
\(u\in C_c^\infty(X)\),
\begin{equation}\label{eq:mckean-square}
\begin{split}
  \int_X\left(|\nabla u|^2-\frac{(m-1)^2}{4}u^2\right)\,dV
  ={}&\int_X\left|\nabla u+\frac{m-1}{2}u\nabla B\right|^2\,dV \\
  &+\frac{m-1}{2}\int_X(\Delta B-(m-1))u^2\,dV.
\end{split}
\end{equation}
\end{lemma}

\begin{proof}
Set
\(
  I=\int_X|\nabla u+\frac{m-1}{2}u\nabla B|^2\,dV.
\)
Since \(u\) has compact support, integration by parts has no boundary
term.  Using \(|\nabla B|=1\) and
\(\nabla(u^2)=2u\nabla u\), we obtain
\begin{align*}
  I
  ={}&\int_X|\nabla u|^2\,dV
  +\frac{m-1}{2}\int_X
     \langle\nabla(u^2),\nabla B\rangle\,dV
  +\frac{(m-1)^2}{4}\int_Xu^2\,dV \\
  ={}&\int_X|\nabla u|^2\,dV
  -\frac{m-1}{2}\int_X(\Delta B)u^2\,dV
  +\frac{(m-1)^2}{4}\int_Xu^2\,dV \\
  ={}&\int_X\left(
    |\nabla u|^2-\frac{(m-1)^2}{4}u^2
  \right)\,dV \\
  &-\frac{m-1}{2}
  \int_X\bigl(\Delta B-(m-1)\bigr)u^2\,dV.
\end{align*}
Adding the final integral to both sides gives
\eqref{eq:mckean-square}.
\end{proof}

Lemma~\ref{lem:mckean-square} is the key analytic identity of the paper.
Apply it with \(B=B_\xi\).  Proposition~\ref{prop:busemann-comparison}
gives \(|\nabla B_\xi|=1\) and
\(\Delta B_\xi-(m-1)\geq0\).  Thus
\eqref{eq:mckean-square} decomposes the excess in McKean's inequality into
two nonnegative terms and therefore directly implies
\eqref{eq:mckean-bound}.  Moreover, if equality holds in McKean's
inequality, then, for every fixed \(\xi\), any \(L^2\)-normalized
minimizing sequence makes both terms tend to zero simultaneously.  These
are precisely the first and second defect limits
\eqref{eq:first-defect} and \eqref{eq:second-defect}.

We call the squared \(L^2\)-term the \emph{first defect}; it measures the
first-order error in
\[
  \nabla u+\frac{m-1}{2}u\nabla B_\xi=0.
\]
We call the integral containing
\(\Delta B_\xi-(m-1)\) the \emph{second defect}; it is the
\(u^2\,dV\)-weighted error in the Busemann Laplacian equality
\[
  \Delta B_\xi=m-1.
\]

\subsection{Vanishing of the first and second defects}

The next proposition makes the preceding consequence of
Lemma~\ref{lem:mckean-square} explicit.

\begin{proposition}
\label{prop:defects}
Let \((X^m,g_X)\) be a complete simply connected Riemannian manifold with
\(m\geq2\) and \(-A^2\leq\sec_{g_X}\leq-1\) for some \(A\geq1\).  Assume
\[
  \lambda_1(X)=\frac{(m-1)^2}{4}.
\]
Then there exists an \(L^2\)-normalized minimizing sequence
\(\{u_j\}\subset C_c^\infty(X)\) such that
\begin{equation}\label{eq:minimizing-sequence}
  \int_Xu_j^2\,dV=1,
  \quad
  \int_X|\nabla u_j|^2\,dV
  \longrightarrow\frac{(m-1)^2}{4}.
\end{equation}
For every fixed \(\xi_0\in\partial_\infty X\), the same sequence satisfies
the following two conclusions.
\begin{enumerate}[label=\textup{(\roman*)}]
\item The first defect vanishes:
\begin{equation}\label{eq:first-defect}
  \int_X\left|
    \nabla u_j+\frac{m-1}{2}u_j\nabla B_{\xi_0}
  \right|^2\,dV\longrightarrow0.
\end{equation}
\item The second defect vanishes:
\begin{equation}\label{eq:second-defect}
  \int_X\bigl(\Delta B_{\xi_0}-(m-1)\bigr)u_j^2\,dV
  \longrightarrow0.
\end{equation}
\end{enumerate}
\end{proposition}

\begin{proof}
By the variational definition of \(\lambda_1(X)\) and the homogeneity of
the Rayleigh quotient, for each \(j\geq1\) we may choose
\(u_j\in C_c^\infty(X)\) such that \(\int_Xu_j^2\,dV=1\) and
\[
  \frac{(m-1)^2}{4}
  \leq\int_X|\nabla u_j|^2\,dV
  \leq\frac{(m-1)^2}{4}+\frac1j.
\]
Set
\[
  \varepsilon_j
  :=\int_X|\nabla u_j|^2\,dV-\frac{(m-1)^2}{4}
  \quad\text{so that}\quad
  0\leq\varepsilon_j\leq\frac1j.
\]
Thus \(\varepsilon_j\to0\), and
\eqref{eq:minimizing-sequence} follows.

Fix \(\xi_0\in\partial_\infty X\).  By
Proposition~\ref{prop:busemann-comparison},
\[
  B_{\xi_0}\in C^2(X),\quad
  |\nabla B_{\xi_0}|=1,\quad
  \Delta B_{\xi_0}\geq m-1.
\]
Applying Lemma~\ref{lem:mckean-square} with \(B=B_{\xi_0}\) gives the exact
decomposition
\[
  \varepsilon_j
  =\int_X\left|\nabla u_j+\frac{m-1}{2}u_j\nabla B_{\xi_0}\right|^2\,dV
   +\frac{m-1}{2}\int_X
    \bigl(\Delta B_{\xi_0}-(m-1)\bigr)u_j^2\,dV.
\]
Proposition~\ref{prop:busemann-comparison} also shows that both terms on the
right are nonnegative.  Hence
\begin{align*}
  0\leq
  \int_X\left|
    \nabla u_j+\frac{m-1}{2}u_j\nabla B_{\xi_0}
  \right|^2\,dV
  &\leq\varepsilon_j\longrightarrow0,\\
  0\leq
  \int_X\bigl(\Delta B_{\xi_0}-(m-1)\bigr)u_j^2\,dV
  &\leq\frac{2\varepsilon_j}{m-1}\longrightarrow0.
\end{align*}
These are precisely \eqref{eq:first-defect} and
\eqref{eq:second-defect}.
\end{proof}

The two vanishing statements play different roles below.  The second defect
will place the limiting measure on the zero set of the Busemann defect,
whereas the first will yield stationarity and then leaf saturation of its
support.  Since the measures \(u_j^2dV\) may escape to infinity on \(X\), we
now pass to the compact horospherical suspension, where their total mass is
preserved.

\section{The horospherical suspension}
\label{sec:suspension}

We return to the cocompact setting of Theorem~\ref{thm:main}.  Let
\((M^m,g)\) be a closed Riemannian manifold with \(\sec_g\leq-1\), and let
\[
  \pi\colon(\wti M,\wti g)\longrightarrow(M,g)
\]
be its universal Riemannian covering.  Since \(M\) is compact, its sectional
curvature is also bounded below.  Thus \((\wti M,\wti g)\) satisfies the
curvature bounds in \eqref{eq:pinched-curvature}, and all the results of
Section~\ref{sec:busemann-defects} apply with \(X=\wti M\).

The next step is to extract a nonzero limit from the probability measures
\(u_j^2\,dV\).  On the noncompact manifold \(\wti M\), their mass may
escape to infinity.  Passing to the compact quotient \(M\) would prevent
this loss, but it would forget the boundary point \(\xi_0\) that defines
the Busemann function and its defect.  This boundary point cannot be held
fixed while the spatial variable is recentered: a deck transformation
\(\gamma\) sends the pair \((x,\xi_0)\) to
\((\gamma x,\gamma\xi_0)\).

The compact space must therefore include both variables and identify them
equivariantly.  The horospherical suspension defined below has these
properties.  The map
\[
  x\longmapsto[x,\xi_0]
\]
pushes \(u_j^2\,dV\) to a probability measure on the compact suspension,
while preserving the Busemann information needed in the equality
argument.

\subsection{The suspension and the unit tangent bundle}

Identify \(\Gamma=\pi_1(M)\) with the deck group of
\(\pi\colon\wti M\to M\).  Each \(\gamma\in\Gamma\) is an isometry of
\(\wti M\) and therefore acts on \(\partial_\infty\wti M\).  We use the
diagonal action
\[
  \gamma\cdot(x,\xi):=(\gamma x,\gamma\xi).
\]
The \emph{horospherical suspension} is the quotient
\begin{equation}\label{eq:suspension}
  Z
  :=
  (\wti M\times\partial_\infty\wti M)/\Gamma.
\end{equation}
The quotient is equipped with the quotient topology, and the class of
\((x,\xi)\) is denoted by \([x,\xi]\).  This is the visual-boundary form
of Ledrappier's horospherical suspension (see
\cite{Ledrappier2010}, Section~1, and
\cite{LedrappierWang2010}, Section~2).

We first identify the space before taking the quotient.  For
\((x,\xi)\in\wti M\times\partial_\infty\wti M\), let \(c_{x,\xi}\) be
the unique unit-speed geodesic ray from \(x\) to \(\xi\), and let
\[
  v_{x,\xi}:=\dot c_{x,\xi}(0)\in S_x\wti M.
\]
Define
\[
  \Phi\colon
  \wti M\times\partial_\infty\wti M
  \longrightarrow S\wti M,
  \quad
  \Phi(x,\xi):=v_{x,\xi}.
\]
Every unit vector determines its base point and its forward endpoint, so
\(\Phi\) is bijective.  Its inverse is
\[
  \Phi^{-1}(v)
  =
  \bigl(p(v),c_v(+\infty)\bigr),
\]
where \(p\colon S\wti M\to\wti M\) is the bundle projection and
\(c_v(t):=\exp_{p(v)}(tv)\).  Proposition~\ref{prop:busemann-comparison}
\textup{(iii)} gives the continuity of
\((x,\xi)\mapsto v_{x,\xi}\), while the cone topology gives the
continuity of \(v\mapsto c_v(+\infty)\).  Hence \(\Phi\) is a
homeomorphism.

The deck group acts on \(S\wti M\) by
\[
  \gamma\cdot v:=d\gamma_x(v),
  \quad v\in S_x\wti M.
\]
Since \(\gamma\) is an isometry, \(\gamma\circ c_{x,\xi}\) is the
unit-speed ray from \(\gamma x\) to \(\gamma\xi\).  Uniqueness gives
\[
  \gamma\circ c_{x,\xi}=c_{\gamma x,\gamma\xi}.
\]
Differentiating at \(t=0\) gives
\[
  v_{\gamma x,\gamma\xi}
  =d\gamma_x(v_{x,\xi}).
\]
Thus \(\Phi\) is \(\Gamma\)-equivariant and induces a homeomorphism
\[
  \Phi_\Gamma\colon
  Z\longrightarrow S\wti M/\Gamma,
  \quad
  \Phi_\Gamma([x,\xi])=[v_{x,\xi}].
\]

We next identify the quotient \(S\wti M/\Gamma\).  Because \(\pi\) is a
local isometry, its differential maps unit vectors to unit vectors and
defines
\[
  \Psi\colon S\wti M/\Gamma\longrightarrow SM,
  \quad
  \Psi([v])=d\pi_x(v),
  \quad v\in S_x\wti M.
\]
This map is well defined: since \(\pi\circ\gamma=\pi\),
\[
  d\pi_{\gamma x}\circ d\gamma_x=d\pi_x.
\]
Every unit vector in \(SM\) lifts to a unit vector in \(S\wti M\), and
any two lifts differ by a deck transformation.  Thus \(\Psi\) is
bijective.  In evenly covered neighborhoods, both \(\Psi\) and its
inverse are represented by the differential of a local isometry, so
\(\Psi\) is a homeomorphism.

Composing the two quotient maps gives
\begin{equation}\label{eq:suspension-unit-tangent}
  \overline\Phi
  :=
  \Psi\circ\Phi_\Gamma
  \colon Z\longrightarrow SM,
  \quad
  \overline\Phi([x,\xi])
  :=
  d\pi_x(v_{x,\xi}).
\end{equation}
Equivalently, the formula is independent of the representative because
\[
  d\pi_{\gamma x}(v_{\gamma x,\gamma\xi})
  =
  d\pi_{\gamma x}\bigl(d\gamma_x(v_{x,\xi})\bigr)
  =
  d\pi_x(v_{x,\xi}).
\]
Both factors in the definition of \(\overline\Phi\) are homeomorphisms;
hence
\[
  Z\cong SM.
\]
Since \(M\) is closed, \(SM\), and hence \(Z\), is compact and metrizable.

\subsection{The suspension as a foliated space}
\label{subsec:suspension-foliated-space}

We begin with the regularity convention used below.  Let
\(r\in\mathbb N\cup\{\infty\}\).  An \(m\)-dimensional
\(C^{r,0}\) \emph{foliated space} is a topological space \(X\) equipped
with charts
\[
  \psi_\alpha\colon U_\alpha\longrightarrow B_\alpha\times T_\alpha,
\]
where \(B_\alpha\subset\mathbb R^m\) is an open ball and \(T_\alpha\) is
a topological space, such that every coordinate change has the form
\[
  (x,\tau)\longmapsto
  \bigl(h_{\beta\alpha}(x,\tau),
        \theta_{\beta\alpha}(\tau)\bigr).
\]
For each fixed \(\tau\), the map
\(x\mapsto h_{\beta\alpha}(x,\tau)\) is \(C^r\), and all its
\(x\)-derivatives through order \(r\) depend continuously on
\((x,\tau)\).  When \(r=\infty\), this condition is imposed at every
finite order.  Thus \(x\) is the differentiable \emph{leaf variable},
whereas \(\tau\) is only a continuous \emph{transverse variable}.

\begin{definition}
\label{def:Ck0}
Let \(k\) be a nonnegative integer, with \(k\leq r\) when
\(r<\infty\).  A function \(f\colon X\to\mathbb R\) is of class
\(C^{k,0}\) if, in every foliated chart, all derivatives
\[
  \partial_x^I(f\circ\psi_\alpha^{-1})(x,\tau),
  \quad |I|\leq k,
\]
exist and are continuous in both variables.  No transverse derivative is
required.  For a leafwise tensor, the same condition is imposed on its
coordinate components.
\end{definition}

The sets
\[
  \psi_\alpha^{-1}(B_\alpha\times\{\tau\}),
  \quad \tau\in T_\alpha,
\]
are the \emph{plaques}.  A \emph{leaf} is a maximal subset connected by
finite chains of intersecting plaques.  The leaf coordinates make every
leaf a \(C^r\) \(m\)-manifold, although its manifold topology may be
finer than the subspace topology inherited from \(X\).  Unless stated
otherwise, every differential operator below acts only along the leaves.

Let
\[
  q\colon
  \wti M\times\partial_\infty\wti M
  \longrightarrow Z
\]
be the quotient map.  Since the diagonal action sends
\[
  \gamma\bigl(\wti M\times\{\xi\}\bigr)
  =
  \wti M\times\{\gamma\xi\},
\]
the image of a slice depends only on the \(\Gamma\)-orbit of \(\xi\).
Set
\[
  \mc L_\xi
  :=q\bigl(\wti M\times\{\xi\}\bigr)
  =\{[x,\xi]:x\in\wti M\},
\]
and define
\[
  q_\xi\colon\wti M\longrightarrow\mc L_\xi,
  \quad
  q_\xi(x):=[x,\xi].
\]
In Ledrappier's notation, \(\mc L_\xi\) is the leaf \(W_\xi\) of the
Busemann lamination; see \cite{Ledrappier2010}, Section~1.
If \(\mc L_\xi\cap\mc L_\eta\neq\varnothing\), then
\([x,\xi]=[y,\eta]\) for some \(x,y\in\wti M\), so
\((y,\eta)=(\gamma x,\gamma\xi)\) for some \(\gamma\in\Gamma\).  Hence
\(\eta=\gamma\xi\).  The converse follows directly from the diagonal
action.  Therefore
\[
  \mc L_\xi=\mc L_\eta
  \quad\Longleftrightarrow\quad
  \eta\in\Gamma\xi,
\]
and consequently
\[
  \mc F:=\{\mc L_\xi:\xi\in\partial_\infty\wti M\}
\]
is a partition of \(Z\).

\begin{lemma}
\label{lem:suspension-foliation}
The quotient \(Z\) is a compact \(m\)-dimensional
\(C^{\infty,0}\) foliated space, and its leaves are the sets
\(\mc L_\xi\).  For every \(\xi\in\partial_\infty\wti M\), the map
\(q_\xi\colon\wti M\to\mc L_\xi\) is a Riemannian covering with deck
group
\(\Gamma_\xi=\{\gamma\in\Gamma:\gamma\xi=\xi\}\), and hence induces a
Riemannian isometry
\(\wti M/\Gamma_\xi\to\mc L_\xi\).  In particular, every leaf
\(\mc L_\xi\) is a smooth, connected, complete Riemannian manifold.
\end{lemma}

\begin{proof}
\emph{Foliated atlas.}
Let \(U\subset M\) be an evenly covered coordinate ball, and choose one
lift \(\wti U\subset\wti M\).  Set
\[
  Z_{\wti U}:=q\bigl(\wti U\times\partial_\infty\wti M\bigr).
\]
The quotient map \(q\) is open, and distinct deck translates of
\(\wti U\) are disjoint.  Hence \(Z_{\wti U}\) is open in \(Z\), and
\(q\) restricts to a homeomorphism from
\(\wti U\times\partial_\infty\wti M\) onto \(Z_{\wti U}\).  Choose
smooth coordinates
\(\varphi_U\colon\wti U\to B_U\subset\mathbb R^m\), and set
\[
  \psi_{\wti U}\colon Z_{\wti U}
  \longrightarrow B_U\times\partial_\infty\wti M,
  \quad
  \psi_{\wti U}([x,\xi])
  :=
  \bigl(\varphi_U(x),\xi\bigr).
\]
This is a homeomorphism.  As \(U\) ranges over an evenly covered
coordinate cover of \(M\), these charts cover \(Z\), since every class
has a representative in one of the chosen lifts.  Their plaques are the
sets \(q(\wti U\times\{\xi\})\), with
\(\xi\in\partial_\infty\wti M\).

An overlap of two such charts decomposes into open pieces on each of
which the two representatives are related by a fixed
\(\gamma\in\Gamma\).  On such a piece the coordinate change is
\[
  (x,\xi)
  \longmapsto
  \left(
    \varphi_V\circ\gamma\circ\varphi_U^{-1}(x),
    \gamma\xi
  \right).
\]
The first component is smooth and independent of \(\xi\), while the
second is the boundary homeomorphism
\(\xi\mapsto\gamma\xi\).  Thus these charts form a
\(C^{\infty,0}\) foliated atlas of leaf dimension \(m\).

\smallskip
\noindent\emph{Identification of the leaves.}
In the chart determined by \(\wti U\), write
\(P(\wti U,\eta):=q(\wti U\times\{\eta\})\subset\mc L_\eta\).
If \(P(\wti U,\eta)\) and \(P(\wti V,\zeta)\) intersect, then
\([u,\eta]=[v,\zeta]\) for some \(u\in\wti U\) and \(v\in\wti V\).
It follows that
\((v,\zeta)=(\gamma u,\gamma\eta)\) for some \(\gamma\in\Gamma\).
In particular, \(\zeta=\gamma\eta\), and hence
\(\mc L_\eta=\mc L_\zeta\).  Therefore every plaque chain starting at
\([x,\xi]\) is contained in \(\mc L_\xi\).

Conversely, let \([y,\xi]\in\mc L_\xi\), and choose a path
\(c\colon[0,1]\to\wti M\) from \(x\) to \(y\).  For every
\(t_0\in[0,1]\), there are an interval \(I_{t_0}\), a chart
\(Z_{\wti U}\), and \(\gamma\in\Gamma\) such that
\(\gamma c(I_{t_0})\subset\wti U\).  In that chart the transverse
coordinate of \(q_\xi\circ c\) on \(I_{t_0}\) is the fixed point
\(\gamma\xi\), so this part of the path lies in one plaque.  Compactness
of \([0,1]\) gives a finite subdivision into such subarcs.  Their plaques
meet successively and form a chain from \([x,\xi]\) to \([y,\xi]\).
Therefore
\(\operatorname{Leaf}([x,\xi])=\mc L_\xi\).

\smallskip
\noindent\emph{Leaf metrics and completeness.}
Use the chosen lift \(\wti U\) to give each plaque the restriction of
\(\wti g\).  This definition is independent of the chart because every
leafwise coordinate change is induced by a deck isometry.  It therefore
defines a smooth Riemannian metric on every leaf, and
\(q_\xi\colon\wti M\to\mc L_\xi\) is a local isometry.  Moreover,
\[
  q_\xi(x)=q_\xi(y)
  \quad\Longleftrightarrow\quad
  y=\gamma x
  \quad\text{for some }\gamma\in
  \Gamma_\xi:=\{\gamma\in\Gamma:\gamma\xi=\xi\}.
\]
The subgroup \(\Gamma_\xi\) acts freely and properly discontinuously by
isometries.  Hence \(\wti M\to\wti M/\Gamma_\xi\) is a Riemannian
covering.  The displayed description of the fibers shows that \(q_\xi\)
factors through a bijective local isometry
\(\overline q_\xi\colon\wti M/\Gamma_\xi\to\mc L_\xi\), which is
therefore a Riemannian isometry.  Consequently, \(q_\xi\) is a
Riemannian covering with deck group \(\Gamma_\xi\).
The leaf \(\mc L_\xi\) is connected because \(\wti M\) is connected,
and it is complete because each of its geodesics lifts to the complete
manifold \((\wti M,\wti g)\).

Finally, \eqref{eq:suspension-unit-tangent} identifies \(Z\) with \(SM\),
which is compact because \(M\) is closed.
\end{proof}

\begin{unnumberedremark}
A leaf should not be confused with the base manifold \(M\).  The
universal covering induces a well-defined projection
\[
  \Pi\colon Z\longrightarrow M,
  \quad
  \Pi([x,\eta])=\pi(x).
\]
Under \(Z\cong SM\), this is the footpoint projection.  Its restriction
\(\Pi_\xi\colon\mc L_\xi\to M\) is the covering associated with the
subgroup \(\Gamma_\xi\subseteq\Gamma\), and
\[
  \Pi_\xi\circ q_\xi=\pi,
  \qquad
  \mc L_\xi\cong\wti M/\Gamma_\xi,
  \qquad
  M=\wti M/\Gamma.
\]

Here \(\Gamma\) is a torsion-free non-elementary word-hyperbolic group.
The stabilizer of a boundary point satisfies
\[
  \Gamma_\xi=\{e\}
  \quad\text{or}\quad
  \Gamma_\xi\cong\mathbb Z
\]
(see \cite{Bowditch1998}).  Since \(\Gamma\) is not virtually
cyclic, \(\Gamma_\xi\) has infinite index.  Thus \(\Pi_\xi\) is
infinite-sheeted and every leaf is noncompact: it is isometric to
\(\wti M\) when \(\Gamma_\xi=\{e\}\), and otherwise to a cyclic quotient
of \(\wti M\).

This does not conflict with the compactness of \(Z\).  The intrinsic
topology of a leaf may be finer than its subspace topology, so a
noncompact leaf may be nonproper or even dense in \(Z\), just as an
irrational line is dense in a torus.  Intrinsically, however, every leaf
is smooth, connected, and complete, and it inherits the local curvature
bounds of \(\wti M\).  These are precisely the properties used in the
leafwise elliptic and heat-kernel arguments below.
\end{unnumberedremark}

\begin{proposition}
\label{prop:busemann-field-regularity}
Let \((M^m,g)\) be a closed Riemannian manifold with \(\sec_g\leq-1\), and
let
\[
  Z=(\wti M\times\partial_\infty\wti M)/\Gamma.
\]
The Busemann vector field on the covering space,
\[
  \widetilde V(x,\xi):=\nabla B_\xi(x)
\]
is \(\Gamma\)-equivariant and therefore descends to a leafwise unit vector
field \(V\) on \(Z\).  The maps
\[
  (x,\xi)\longmapsto\widetilde V(x,\xi),
  \quad
  (x,\xi)\longmapsto\nabla_x\widetilde V(x,\xi)
\]
are continuous on \(\wti M\times\partial_\infty\wti M\).  Moreover, for
each fixed \(\xi\), the map \(x\mapsto\widetilde V(x,\xi)\) is smooth.
Consequently,
\[
  V\in C^{1,0}(T\mc F),
  \qquad
  V|_{\mc L_\xi}\in C^\infty(T\mc L_\xi)
  \quad\text{for every }\xi.
\]
\end{proposition}

\begin{proof}
\emph{Equivariance and unit length.}
For \(\gamma\in\Gamma\), the Busemann functions
\(B_{\gamma\xi}\circ\gamma\) and \(B_\xi\) differ by a constant.  Since
\(\gamma\) is an isometry, taking gradients gives
\[
  \nabla B_{\gamma\xi}(\gamma x)
  =d\gamma_x\bigl(\nabla B_\xi(x)\bigr).
\]
Thus \(\widetilde V\) is equivariant and defines \(V\) on the quotient.
Equation~\eqref{eq:unit-busemann-gradient} gives \(|V|=1\).

\smallskip
\noindent\emph{Joint continuity to first leafwise order.}
Since \(M\) is compact, there is \(A\geq1\) such that
\[
  -A^2\leq\sec_g\leq-1.
\]
Proposition~\ref{prop:busemann-comparison}\textup{(iii)} therefore
applies to \(\wti M\) and gives joint continuity of \(\widetilde V\)
and its covariant derivative in the \(x\)-variable:
\[
  \nabla_x\widetilde V=(\nabla^2B_\xi)^\sharp.
\]
In a flow box from Lemma~\ref{lem:suspension-foliation}, let \(V^i\)
denote the leafwise coordinate components.  Then
\[
  \partial_jV^i
  =
  (\nabla_jV)^i-\Gamma^i_{jk}V^k.
\]
The Christoffel symbols in these lifted coordinates are smooth in the
leaf variable and independent of the boundary coordinate.  Hence both
\(V^i\) and \(\partial_jV^i\) are jointly continuous, which is precisely
\(V\in C^{1,0}(T\mc F)\).

\smallskip
\noindent\emph{Smoothness for fixed \(\xi\).}
Before taking the quotient, \(\Phi\) sends the slice
\(\wti M\times\{\xi\}\) onto
\[
  W^{ws}(\xi)
  =
  \{v_{x,\xi}:x\in\wti M\}
  \subset S\wti M.
\]
Two unit vectors lie in the same weak-stable leaf exactly when their
forward geodesic rays have the same endpoint at infinity.  The adjective
``weak'' indicates that a shift in the geodesic-flow parameter is
allowed; fixing also the Busemann level gives the strong-stable leaves.
Under \(\overline\Phi\colon Z\to SM\), the leaf \(\mc L_\xi\)
corresponds to the immersed weak-stable leaf
\(W^{ws}(\xi)/\Gamma_\xi\).

The geodesic flow on \(SM\) is a smooth Anosov flow
(see \cite{Anosov1967}).  By the stable manifold theorem, its weak-stable
leaves are smooth immersed submanifolds (see
\cite{HirschPughShub1977}, Theorem~4.1\textup{(e)} and the remarks
following it); their transverse dependence is generally only H\"older.
The lifted weak-stable foliation on \(S\wti M\) has the same leafwise
regularity.  Since there is a unique ray from each \(x\) to \(\xi\), the
footpoint projection
\(p\colon S\wti M\to\wti M\) restricts to a bijection
\[
  p|_{W^{ws}(\xi)}\colon W^{ws}(\xi)\longrightarrow\wti M.
\]

To prove that this bijection is a diffeomorphism, fix
\(v=v_{x,\xi}\).  The tangent space of the weak-stable leaf is
\[
  T_vW^{ws}(\xi)=\mathbb RX(v)\oplus E^s(v),
\]
where \(X\) is the geodesic vector field, and \(dp_v(X(v))=v\).
A vector \(z\in E^s(v)\) is represented by a normal Jacobi field \(J_z\)
with \(dp_v(z)=J_z(0)\) and \(J_z(t)\to0\).  If \(J_z(0)=0\), then the
Jacobi equation and \(\sec_g\leq-1\) show that
\(\frac12|J_z|^2\) is strictly convex unless \(J_z\equiv0\).  Since its
value and first derivative vanish at \(t=0\), strict convexity would make
it increase and prevent \(J_z(t)\to0\).  Thus \(J_z\equiv0\).
Furthermore, \(dp_v(E^s(v))\perp v\), whereas \(dp_v(X(v))=v\).
Consequently, \(dp_v\) is injective on
\(\mathbb RX(v)\oplus E^s(v)\), and hence is an isomorphism by dimension.
The inverse function theorem now shows that
\(p|_{W^{ws}(\xi)}\) is a local diffeomorphism.  Since it is also
bijective, it is a diffeomorphism.  Its inverse
\[
  \wti M\longrightarrow W^{ws}(\xi),
  \quad
  x\longmapsto v_{x,\xi},
\]
is therefore smooth.  Finally, \eqref{eq:unit-busemann-gradient} gives
\[
  \nabla B_\xi(x)=-v_{x,\xi},
\]
so \(x\mapsto\widetilde V(x,\xi)\) is smooth.  The leaf charts from
Lemma~\ref{lem:suspension-foliation} then show that
\(V|_{\mc L_\xi}\) is smooth.
\end{proof}

\begin{unnumberedremark}
The two regularity statements serve different purposes.
\(V\in C^{1,0}\) means that \(V\) and its first leafwise derivative vary
continuously in the transverse parameter; it does not require any
transverse derivative.  Smoothness of \(V|_{\mc L_\xi}\) gives
derivatives of every order on a fixed leaf, without asserting that the
higher derivatives vary continuously with \(\xi\).  The first property
provides continuous coefficients for the global leafwise diffusion; the
second permits classical heat-kernel arguments on an individual leaf.
\end{unnumberedremark}

\subsection{Leafwise differential calculus}

All derivatives in this subsection are taken only in the leaf variable.
For \(\xi\in\partial_\infty\wti M\), let
\[
  \iota_\xi\colon\wti M\longrightarrow\mc L_\xi,
  \quad \iota_\xi(x)=[x,\xi].
\]
Lemma~\ref{lem:suspension-foliation} identifies this map with the
Riemannian covering
\[
  \wti M\longrightarrow
  \wti M/\Gamma_\xi\cong\mc L_\xi.
\]

\smallskip
\noindent\emph{Leafwise regularity.}
For a function \(F\colon Z\to\mathbb R\), set
\[
  F_\xi:=F\circ\iota_\xi\colon\wti M\longrightarrow\mathbb R.
\]
For an integer \(k\geq0\), we say that \(F\) is \(C^k\)
\emph{along the leaves} if \(F|_{\mc L}\) is \(C^k\) for every leaf
\(\mc L\).  Since \(\iota_\xi\) is a covering, this is equivalent to
\[
  F_\xi\in C^k(\wti M)
  \quad
  \text{for every }\xi\in\partial_\infty\wti M.
\]
This condition concerns one leaf at a time and imposes no continuity as
the leaf varies.

The notation \(F\in C^{k,0}(Z)\) is understood in the sense of
Definition~\ref{def:Ck0}.  It is stronger than being \(C^k\) along
each leaf separately, because the leafwise derivatives must also vary
continuously in the transverse parameter.  As before, no transverse
derivative is required.

\smallskip
\noindent\emph{Leafwise differential operators.}
If \(F\) is \(C^1\) along the leaves, the leafwise metric defines
\(\nabla^{\mathrm{leaf}}F\) by restricting \(F\) to each leaf.  If
\(F\) is \(C^2\) along the leaves, define
\(\Delta^{\mathrm{leaf}}F\) in the same way.  The covering
\(\iota_\xi\) is a local isometry, so the naturality of the gradient and
Laplacian gives
\begin{equation}\label{eq:leafwise-lift-calculus}
\begin{aligned}
  d(\iota_\xi)_x\bigl(\nabla F_\xi(x)\bigr)
  &=\nabla^{\mathrm{leaf}}F(\iota_\xi(x)),\\
  \Delta F_\xi
  &=(\Delta^{\mathrm{leaf}}F)\circ\iota_\xi,
\end{aligned}
\end{equation}
where the unadorned \(\nabla\) and \(\Delta\) are the Riemannian operators on
\((\wti M,\wti g)\).  If \(dV_{\mc L_\xi}\) denotes the Riemannian
volume density of the leaf, then
\[
  \iota_\xi^*dV_{\mc L_\xi}=dV_{\wti g}.
\]
Consequently, a pointwise leafwise differential identity may be checked
after pulling it back to \(\wti M\) with \(\xi\) fixed.  Integration by
parts may likewise be carried out in plaque coordinates, or intrinsically
on a leaf under the usual compact-support or integrability hypotheses.

\section{The Busemann defect and the limiting measure}
\label{sec:limiting-measure}

Pushing the quadratic mass of a minimizing sequence to the compact
suspension \(Z\) prevents loss of mass.  The Busemann defect descends to
\(Z\), and the limiting measure is supported in its zero set.

Define the \emph{Busemann defect} on the covering space by
\[
  \widetilde D(x,\xi)=\Delta B_\xi(x)-(m-1),
  \quad (x,\xi)\in\wti M\times\partial_\infty\wti M.
\]
By Proposition~\ref{prop:busemann-comparison}, \(\widetilde D\) is
continuous and nonnegative.

For \(\gamma\in\Gamma\), the functions
\(B_{\gamma\xi}\circ\gamma\) and \(B_\xi\) differ by a constant.  Since
\(\gamma\) is an isometry,
\[
  \Delta B_{\gamma\xi}(\gamma x)=\Delta B_\xi(x).
\]
Thus \(\widetilde D\) is \(\Gamma\)-invariant and descends to the
continuous function
\begin{equation}\label{eq:busemann-defect}
  D\colon Z\longrightarrow[0,\infty),
  \quad
  D([x,\xi])=\Delta B_\xi(x)-(m-1)
\end{equation}
for every \((x,\xi)\in\wti M\times\partial_\infty\wti M\).  The set
\(D^{-1}(0)\) consists of the classes represented by pairs
\((x,\xi)\) for which \(\Delta B_\xi(x)=m-1\).

To construct the limiting measure, assume
\(\lambda_1(\wti M)=(m-1)^2/4\), and fix
\(\xi_0\in\partial_\infty\wti M\).  Choose a normalized minimizing sequence
\(u_j\) as in Proposition~\ref{prop:defects}, with \(X=\wti M\).  The
measure \(u_j^2dV\) is a probability measure on \(\wti M\).  Push it to
the compact suspension under the continuous map
\[
  \iota_{\xi_0}\colon\wti M\longrightarrow Z,
  \quad x\longmapsto[x,\xi_0],
\]
and define
\begin{equation}\label{eq:mu-j}
  \mu_j=(\iota_{\xi_0})_*(u_j^2dV).
\end{equation}
Each \(\mu_j\) is a probability measure, and for \(f\in C(Z)\),
\begin{equation}\label{eq:mu-j-testing}
  \int_Z f\,d\mu_j
  =\int_{\wti M}f([x,\xi_0])\,u_j(x)^2\,dV(x).
\end{equation}

\begin{proposition}
\label{prop:defect-support}
After passing to a subsequence, \(\mu_{j_k}\rightharpoonup\mu\), where
\(\mu\) is a probability measure satisfying
\[
  \int_ZD\,d\mu=0,
  \qquad
  \supp\mu\subseteq D^{-1}(0).
\]
\end{proposition}

\begin{proof}
The probability measures on the compact metrizable space \(Z\) are
weakly sequentially compact (see \cite{Bogachev2007},
Theorem~8.6.7).  Thus some subsequence converges weakly to a probability
measure \(\mu\).  Continuity of \(D\), weak convergence, the pushforward
formula \eqref{eq:mu-j-testing}, and the definition
\eqref{eq:busemann-defect} give
\[
  \int_ZD\,d\mu
  =\lim_{k\to\infty}\int_ZD\,d\mu_{j_k}
  =\lim_{k\to\infty}
    \int_{\wti M}\bigl(\Delta B_{\xi_0}-(m-1)\bigr)u_{j_k}^2\,dV
  =0.
\]
The last equality is \eqref{eq:second-defect}.  Since \(D\geq0\),
the identity forces \(D=0\) on \(\supp\mu\), which proves the stated
support containment.
\end{proof}

\section{Leafwise stationarity and support saturation}
\label{sec:stationarity}

Fix a weak subsequential limit \(\mu\) as in
Proposition~\ref{prop:defect-support}.  The first defect gives
infinitesimal stationarity for a leafwise drift operator.  The general
results of Section~\ref{sec:candel-semigroup-construction} then imply
stationarity and leaf saturation of \(\supp\mu\).

\subsection{The leafwise drift operator}

Let \(V\) be the Busemann field from
Proposition~\ref{prop:busemann-field-regularity}.  It is \(C^{1,0}\),
smooth on each leaf, and its lift to the covering
\(\iota_\xi\colon\wti M\to\mc L_\xi\) is \(\nabla B_\xi\).

Let \(f\in C(Z)\), and assume that \(f|_{\mc L}\) is \(C^2\) on every leaf
\(\mc L\).  Define
\begin{equation}\label{eq:leafwise-L}
  Lf
  :=\Delta^{\mathrm{leaf}}f
    -(m-1)\langle V,\nabla^{\mathrm{leaf}}f\rangle.
\end{equation}
By \eqref{eq:leafwise-lift-calculus}, its pullback to \(\wti M\) is
\begin{equation}\label{eq:leafwise-L-lift}
  (Lf)\circ\iota_\xi
  =\Delta(f\circ\iota_\xi)
   -(m-1)\bigl\langle\nabla B_\xi,
                 \nabla(f\circ\iota_\xi)\bigr\rangle.
\end{equation}
Set \(A:=L/2\).  We use its classical continuous domain
\begin{equation}\label{eq:generator-test-domain}
  \mc D_0(A)
  :=\left\{f\in C(Z):
    f|_{\mc L}\in C^2(\mc L)\text{ for every leaf }\mc L,
    \quad Af\in C(Z)
  \right\}.
\end{equation}
Since \(A=L/2\), the condition \(Af\in C(Z)\) is equivalent to
\(Lf\in C(Z)\).  This is the classical domain used in
Section~\ref{sec:candel-semigroup-construction} and also the test domain
in Suzaki's definition of an \(A\)-harmonic measure (see
\cite{Suzaki2015}, p.~256).

\begin{lemma}
\label{lem:leafwise-L-elliptic}
In every quotient product chart, the principal coefficients of \(L\) are
of class \(C^{\infty,0}\), its first-order coefficients are of
class \(C^{1,0}\), and all coefficients are smooth on each plaque.  Its
principal symbol is \((g^{\mathrm{leaf}})^{-1}\), and it has no
zeroth-order term.  In particular, \(L\) is uniformly elliptic along the
leaves and \(L1=0\).
\end{lemma}

\begin{proof}
In a foliated chart \(B\times T\), write
\[
  g^{\mathrm{leaf}}=g_{ij}(y,\tau)\,dy^i\,dy^j,
  \quad
  V=V^j(y,\tau)\partial_j.
\]
The standard coordinate formula for the Laplace--Beltrami operator gives
\[
  Lf
  =g^{ij}\partial_i\partial_jf
   +\left[
      \frac1{\sqrt{\det g}}\,
      \partial_i\bigl(\sqrt{\det g}\,g^{ij}\bigr)
      -(m-1)V^j
    \right]\partial_jf.
\]
The leafwise metric is of class \(C^{\infty,0}\), while
Proposition~\ref{prop:busemann-field-regularity} gives
\(V\in C^{1,0}\) and shows that \(V\) is smooth on each plaque.  Hence
\(g^{ij}\in C^{\infty,0}\), the displayed first-order coefficients belong
to \(C^{1,0}\), and all coefficients are smooth on each plaque.  The
principal matrix is \((g^{ij})\), so the principal symbol is
\((g^{\mathrm{leaf}})^{-1}\).  Its positive definiteness and the
compactness of \(Z\) give, after choosing a finite foliated atlas, a
constant \(\theta\in(0,1]\) such that
\[
  \theta|\zeta|^2
  \leq g^{ij}(y,\tau)\zeta_i\zeta_j
  \leq\theta^{-1}|\zeta|^2.
\]
Finally, the displayed formula contains no zeroth-order term, and hence
\(L1=0\).
\end{proof}

Lemma~\ref{lem:classical-gradient-control}, applied with
\(\ms A=A\), shows that every \(f\in\mc D_0(A)\) has a continuous,
uniformly bounded leafwise gradient.  This is the only domain
regularity needed in the first-defect argument.

\subsection{The leafwise diffusion}

The general results of Section~\ref{sec:candel-semigroup-construction}
apply to \(A=L/2\).  We state their consequences for the suspension.

\begin{lemma}
\label{lem:saturated-support}
The maximal continuous realization of \(A\) generates a conservative
Feller Markov semigroup \(\{P_t\}_{t\geq0}\) on \(C(Z)\).  Moreover:
\begin{enumerate}[label=\textup{(\roman*)}]
\item On every leaf, \(P_t\) is the intrinsic minimal heat semigroup of
\(A\).  Its heat kernel is smooth, strictly positive, and conservative
with respect to the principal metric \(g_A=2g^{\mathrm{leaf}}\).

\item A Borel probability measure \(\nu\) on \(Z\) satisfies
\[
  \int_ZP_tf\,d\nu=\int_Zf\,d\nu
  \quad(f\in C(Z),\ t\geq0)
\]
if and only if
\begin{equation}\label{eq:infinitesimal-stationarity}
  \int_ZLf\,d\nu=0
  \quad\text{for every }f\in\mc D_0(A).
\end{equation}

\item If \(\nu\) satisfies the equivalent conditions in
\textup{(ii)}, then
\[
  z\in\supp\nu\quad\Longrightarrow\quad
  \mc L(z)\subset\supp\nu.
\]
\end{enumerate}
\end{lemma}

\begin{proof}
By Lemma~\ref{lem:suspension-foliation}, \(Z\) is a compact
\(C^{\infty,0}\) foliated space.  Lemma~\ref{lem:leafwise-L-elliptic}
verifies all coefficient hypotheses for \(A\), and its principal matrix
is \((g^{ij}/2)\); hence \(g_A=2g^{\mathrm{leaf}}\).
Part~\textup{(i)} follows from
Theorem~\ref{thm:candel-diffusion}\textup{(ii)}--\textup{(iii)}.
Since \(L=2A\), part~\textup{(ii)} is
Proposition~\ref{prop:stationarity-equivalence}, and part~\textup{(iii)}
is Proposition~\ref{prop:stationary-support-saturated}.
\end{proof}

\subsection{Stationarity from the first defect}

We pass the first-defect identity to the weak limit.  Differentiating
\(u_j^2\) gives exactly the coefficient \(m-1\) appearing in \(L\).

\begin{proposition}
\label{prop:stationarity-defect}
Every weak subsequential limit \(\mu\) from
Proposition~\ref{prop:defect-support} satisfies
\begin{equation}\label{eq:mu-infinitesimal-stationarity}
  \int_ZLf\,d\mu=0
  \quad\text{for every }f\in\mc D_0(A).
\end{equation}
Moreover, its support is leaf-saturated; explicitly,
\[
  z\in\supp\mu\quad\Longrightarrow\quad
  \mc L(z)\subset\supp\mu.
\]
\end{proposition}

\begin{proof}
Use the notation of Section~\ref{sec:limiting-measure}, and choose a
subsequence such that \(\mu_{j_k}\rightharpoonup\mu\).  Fix
\(f\in\mc D_0(A)\), and set
\[
  f_0:=f\circ\iota_{\xi_0},
  \quad
  R_j=\nabla u_j+\frac{m-1}{2}u_j\nabla B_{\xi_0}.
\]
The local isometry
\(\iota_{\xi_0}\colon\wti M\to\mc L_{\xi_0}\) gives
\(f_0\in C^2(\wti M)\) and
\[
  |\nabla f_0(x)|
  =|\nabla^{\mathrm{leaf}}f(\iota_{\xi_0}(x))|
  \quad\text{for every }x\in\wti M.
\]
Lemma~\ref{lem:classical-gradient-control}, together with
\(g_A=2g^{\mathrm{leaf}}\), therefore gives
\[
  \|\nabla f_0\|_{L^\infty(\wti M)}
  \leq\|\nabla^{\mathrm{leaf}}f\|_{L^\infty(Z)}<\infty.
\]

By the pushforward identity \eqref{eq:mu-j-testing} and the lifted formula
\eqref{eq:leafwise-L-lift},
\[
  \int_ZLf\,d\mu_j
  =\int_{\wti M}u_j^2
   \bigl(\Delta f_0
   -(m-1)\langle\nabla B_{\xi_0},\nabla f_0\rangle\bigr)\,dV.
\]
Although \(f_0\) need not have compact support, \(u_j\) does, so integration
by parts gives
\[
  \int_{\wti M}u_j^2\Delta f_0\,dV
  =-2\int_{\wti M}u_j
       \langle\nabla u_j,\nabla f_0\rangle\,dV.
\]
Substitution yields
\begin{equation}\label{eq:first-defect-pairing}
  \int_ZLf\,d\mu_j
  =-2\int_{\wti M}u_j\langle R_j,\nabla f_0\rangle\,dV.
\end{equation}
By Cauchy--Schwarz,
\begin{align*}
  \left|\int_ZLf\,d\mu_j\right|
  &\leq2\|u_j\|_{L^2(\wti M)}
        \|R_j\|_{L^2(T\wti M)}
        \|\nabla f_0\|_{L^\infty(\wti M)}\\
  &\leq2\|\nabla^{\mathrm{leaf}}f\|_{L^\infty(Z)}
        \|R_j\|_{L^2(T\wti M)}
  \longrightarrow0.
\end{align*}
The convergence follows from \eqref{eq:first-defect}.  Since
\(Lf\in C(Z)\), weak convergence gives
\[
  \int_ZLf\,d\mu
  =\lim_{k\to\infty}\int_ZLf\,d\mu_{j_k}=0.
\]
Thus \eqref{eq:mu-infinitesimal-stationarity} holds.  By
Lemma~\ref{lem:saturated-support}\textup{(ii)}--\textup{(iii)},
\(\mu\) is stationary and \(\supp\mu\) is leaf-saturated.
\end{proof}

\subsection{From saturated support to Busemann equality}

Propositions~\ref{prop:defect-support} and
\ref{prop:stationarity-defect} show that \(\supp\mu\) lies in
\(D^{-1}(0)\) and is leaf-saturated.  Choosing a leaf contained in this
support therefore gives a single boundary point for which the Busemann
defect vanishes at every base point.

\begin{proposition}
\label{prop:equality-busemann}
Let \((M^m,g)\) be a closed connected Riemannian manifold with \(m\geq2\)
and \(\sec_g\leq-1\), and let \((\wti M,\wti g)\) be its universal cover.
Assume
\[
  \lambda_1(\wti M)=\frac{(m-1)^2}{4}.
\]
Then there exists \(\xi_*\in\partial_\infty\wti M\) such that
\(B_{\xi_*}\in C^2(\wti M)\) and
\[
  |\nabla B_{\xi_*}(x)|=1,
  \quad
  \Delta B_{\xi_*}(x)=m-1
\]
for every \(x\in\wti M\).
\end{proposition}

\begin{proof}
Since \(M\) is compact, there is a constant \(A\geq1\) such that
\[
  -A^2\leq\sec_g\leq-1.
\]
The same bounds hold on \((\wti M,\wti g)\), so
Proposition~\ref{prop:busemann-comparison} shows that every Busemann
function \(B_\xi\) is of class \(C^2\) and satisfies
\(|\nabla B_\xi|=1\).  We seek a boundary point for which the Laplacian
equality holds throughout \(\wti M\).

Proposition~\ref{prop:defect-support} gives a probability measure
\(\mu\), and Proposition~\ref{prop:stationarity-defect} applies to this
measure.  Together they give
\[
  \supp\mu\subseteq D^{-1}(0),
  \quad
  \supp\mu\text{ is leaf-saturated}.
\]

Since \(\mu\) is a probability measure, \(\supp\mu\neq\varnothing\).
Choose \(z_*=[x_*,\xi_*]\in\supp\mu\).  By the definition of the leaf
through \(z_*\) and leaf saturation, for every \(x\in\wti M\),
\[
  [x,\xi_*]\in\mc L(z_*)
  \subseteq\supp\mu
  \subseteq D^{-1}(0).
\]
Therefore, by \eqref{eq:busemann-defect},
\[
  0=D([x,\xi_*])
   =\Delta B_{\xi_*}(x)-(m-1)
  \quad\text{for every }x\in\wti M.
\]
Hence \(\Delta B_{\xi_*}=m-1\) throughout \(\wti M\).  The regularity and
unit-gradient identity were established above.
\end{proof}

It remains to convert this Busemann equality into a global geometric
description of \(\wti M\).

\section{From Busemann equality to hyperbolic rigidity}
\label{sec:rigidity}

In this section we pass from the Busemann equality obtained in
Proposition~\ref{prop:equality-busemann} to hyperbolic rigidity and complete
the proof of Theorem~\ref{thm:main}.  The next lemma contains the
Riemannian part of the argument: the upper curvature bound gives the
Busemann Hessian identity, integrating the gradient flow gives a
warped-product splitting, and the lower curvature bound forces the
reference horosphere to be flat.  We then apply the lemma to the universal
cover of \(M\).  The final subsection explains the two distinct places
where cocompactness is used.

\subsection{Riemannian rigidity and completion of the proof}

\begin{lemma}
\label{lem:riemannian-rigidity}
Let \((X^m,g_X)\) be a complete simply connected Riemannian manifold with
\(m\geq2\), and suppose that, for some \(A\geq1\),
\[
  -A^2\leq\sec_{g_X}\leq-1.
\]
If a Busemann function \(B\in C^2(X)\) satisfies
\[
  |\nabla B|=1,
  \quad
  \Delta B=m-1,
\]
then \(B\) is smooth,
\begin{equation}\label{eq:hessian-equality}
  \nabla^2B=g_X-dB\otimes dB.
\end{equation}
Moreover, \(H=B^{-1}(0)\) is complete, simply connected, and flat, and
\begin{equation}\label{eq:warped-product}
  (X,g_X)
  \cong
  \bigl(\R\times H,dt^2+\ee^{2t}g_H\bigr).
\end{equation}
Consequently, \((X,g_X)\) is isometric to
\(\mathbb H^m(-1)\).
\end{lemma}

\begin{proof}
The lower Hessian comparison inequality
\eqref{eq:busemann-laplacian-comparison} shows that
\[
  Q:=\nabla^2B-(g_X-dB\otimes dB)
\]
is positive semidefinite.  Its trace is
\[
  \tr_{g_X}Q
  =\Delta B-\bigl(m-|\nabla B|^2\bigr)
  =(m-1)-(m-1)=0.
\]
A positive semidefinite symmetric tensor with zero trace vanishes, proving
\eqref{eq:hessian-equality}.

The tensor identity also shows that \(B\) is smooth.  Indeed, in local
coordinates,
\[
  \partial_i\partial_jB
  =\Gamma_{ij}^{k}\partial_kB+(g_X)_{ij}
   -(\partial_iB)(\partial_jB).
\]
Starting from \(B\in C^2\), this equation and induction give
\(B\in C^\infty(X)\).  If \(T=\nabla B\), then
\eqref{eq:hessian-equality} is equivalent to
\begin{equation}\label{eq:nabla-T}
  \nabla_YT=Y-\langle Y,T\rangle T
\end{equation}
for every vector field \(Y\).  In particular,
\begin{equation}\label{eq:T-geodesic}
  \nabla_TT=0,
  \quad
  \nabla_YT=Y\quad\text{when }Y\perp T.
\end{equation}
Thus the integral curves of \(T\) are unit-speed geodesics, while the level
sets of \(B\) are totally umbilical with second fundamental form equal to
their induced metric.

Set \(H=B^{-1}(0)\), with its induced metric \(g_H\).  We next integrate
the vector field \(T=\nabla B\).
Since \(|T|=1\) and \(X\) is complete, the flow \(\Phi_t\) of
\(T=\nabla B\) is defined for every \(t\in\R\).  Along each flow line,
\begin{equation}\label{eq:B-along-flow}
  \frac{d}{dt}B(\Phi_t(x))
  =\langle\nabla B,T\rangle=1,
\end{equation}
and hence \(B(\Phi_t(x))=B(x)+t\).  It follows that
\[
  \Phi\colon\R\times H\longrightarrow X,
  \quad
  \Phi(t,y)=\Phi_t(y),
\]
is a diffeomorphism with inverse
\[
  x\longmapsto\bigl(B(x),\Phi_{-B(x)}(x)\bigr).
\]
Since \(X\) is connected, \(H\) is connected.

Let \(g_t\) be the metric on \(H\) obtained by pulling back the induced
metric on \(B^{-1}(t)\) by \(\Phi_t\).  Extend tangent vector fields
\(Y,Z\) on \(H\) by the flow, so that \([T,Y]=[T,Z]=0\).  Using
\eqref{eq:T-geodesic},
\begin{align*}
  \frac{d}{dt}g_t(Y,Z)
  &=\langle\nabla_TY,Z\rangle+\langle Y,\nabla_TZ\rangle \\
  &=\langle\nabla_YT,Z\rangle+\langle Y,\nabla_ZT\rangle \\
  &=2g_t(Y,Z).
\end{align*}
Therefore \(g_t=\ee^{2t}g_H\).  Since \(T\) is unit and orthogonal to the
level sets, this proves \eqref{eq:warped-product}.

The level set \(H\) is closed in \(X\).  An intrinsic Cauchy sequence in
\(H\) is therefore an ambient Cauchy sequence; completeness of \(X\) gives
an ambient limit in \(H\), and local equivalence of the intrinsic and
ambient distances gives intrinsic convergence.  Thus \(H\) is complete.
Finally, \(X\cong\R\times H\) makes \(H\) a deformation retract of \(X\),
which proves the assertion about simple connectivity.

The sectional curvatures of \eqref{eq:warped-product} are as follows.
If \(Y\) is tangent to \(H\), then
\begin{equation}\label{eq:mixed-curvature}
  \sec_{g_X}(\partial_t,Y)=-1.
\end{equation}
If \(\sigma\subset T_yH\) is a two-plane and \(\sigma_t\) is the
corresponding plane tangent to \(\{t\}\times H\), then
\begin{equation}\label{eq:tangential-curvature}
  \sec_{g_X}(\sigma_t)
  =\ee^{-2t}\sec_{g_H}(\sigma)-1.
\end{equation}
These are the standard warped-product formulas for the warping function
\(\ee^t\).

It remains to use the lower curvature bound.  The splitting alone would
allow \(g_H\) to be nonflat and nonpositively curved.
If \(m=2\), then \(H\) is one-dimensional and hence flat.  Suppose
\(m\geq3\), and let \(\sigma\subset T_yH\) be a two-plane.
The tangential curvature formula \eqref{eq:tangential-curvature} and the
upper curvature bound imply
\[
  \ee^{-2t}\sec_{g_H}(\sigma)-1\leq-1,
\]
and hence \(\sec_{g_H}(\sigma)\leq0\).  If
\(\sec_{g_H}(\sigma)=-\kappa<0\), then
\[
  \sec_{g_X}(\sigma_t)=-\kappa\ee^{-2t}-1
  \longrightarrow-\infty
  \quad\text{as }t\longrightarrow-\infty.
\]
This contradicts the lower curvature bound.
Therefore \(\sec_{g_H}(\sigma)=0\) for every \(\sigma\), so \(g_H\) is flat.

Since \(H\) is complete, simply connected, and flat,
\((H,g_H)\cong(\R^{m-1},g_{\mr{Euc}})\).  With \(r=\ee^{-t}\),
\[
  dt^2+\ee^{2t}g_{\mr{Euc}}
  =\frac{dr^2+g_{\mr{Euc}}}{r^2},
\]
which is the upper-half-space metric on \(\mathbb H^m(-1)\).
\end{proof}

\begin{proof}[Proof of Theorem~\ref{thm:main}]
McKean's estimate gives the asserted lower bound.  Assume equality.
Since \(M\) is closed, the sectional curvature of \(g\) is bounded below;
thus, for some \(A\geq1\),
\[
  -A^2\leq\sec_{\wti g}\leq-1.
\]
Proposition~\ref{prop:equality-busemann} gives a Busemann function
\(B\in C^2(\wti M)\) satisfying
\[
  |\nabla B|=1,
  \quad
  \Delta B=m-1.
\]
Lemma~\ref{lem:riemannian-rigidity} therefore yields
\(\wti M\cong\bH^m(-1)\).

Conversely, if \(\wti M\cong\bH^m(-1)\), then
\[
  \lambda_1(\wti M)=\frac{(m-1)^2}{4}.
\]
The upper bound follows by applying radial cutoffs to
\(\exp(-(m-1)r/2)\), and the reverse inequality is McKean's estimate.
The covering projection \(\wti M\to M\) is a local isometry.  Thus
\(\wti M\cong\bH^m(-1)\) is equivalent to \(g\) having constant sectional
curvature \(-1\).
\end{proof}

\subsection{The role of cocompactness}
\label{subsec:role-cocompactness}

Closedness has two concrete consequences in the proof of
Theorem~\ref{thm:main}.  It yields a uniform lower sectional-curvature
bound, which is the hypothesis that forces the reference horosphere to be
flat in Lemma~\ref{lem:riemannian-rigidity}.  It also makes the horospherical
suspension \(Z\cong SM\) compact, so the probability measures associated
with a minimizing sequence cannot lose mass at infinity.  The two
constructions below show that these are separate issues.

\medskip
\noindent\emph{Failure without a lower curvature bound.}
For \(m\geq3\), let \((H^{m-1},g_H)\) be complete, simply connected,
nonflat, and satisfy \(\sec_{g_H}\leq0\).  The warped product
\[
  (X^m,g_X)
  =\bigl(\R\times H,dt^2+\ee^{2t}g_H\bigr)
\]
is complete and simply connected.  It has mixed sectional curvatures
\(-1\), while for a two-plane
\(\sigma\subset T_yH\),
\[
  \sec_{g_X}(\sigma_t)
  =\ee^{-2t}\sec_{g_H}(\sigma)-1.
\]
Thus \(\sec_{g_X}\leq-1\), but \(X\) is not hyperbolic.  Nevertheless,
\[
  \lambda_1(X)=\frac{(m-1)^2}{4}.
\]
McKean's estimate gives the lower bound.  For the reverse inequality, fix
\(0\neq\varphi\in C_c^\infty(H)\), put \(a=(m-1)/2\), and choose
\(\chi_R\in C_c^\infty(\mathbb R)\) supported in \((0,R+1)\), equal to one
on \([1,R]\), and with derivative bounded independently of \(R\).  Since
\(dV_X=\ee^{(m-1)t}dt\,dV_H\), the functions
\[
  u_R(t,y)=\chi_R(t)\ee^{-at}\varphi(y)
\]
satisfy
\[
  \frac{\int_X|\nabla u_R|^2\,dV_X}
       {\int_Xu_R^2\,dV_X}
  =a^2+O(R^{-1}).
\]
Indeed, the denominator is comparable to \(R\), while the cutoff terms
and the \(H\)-gradient term remain bounded.  Hence the Rayleigh quotients
converge to \(a^2\).  Negative curvature in \(g_H\) is amplified as
\(t\to-\infty\), so \(\sec_{g_X}\) has no uniform lower bound.  This is
precisely the obstruction excluded by the two-sided curvature hypothesis
in Lemma~\ref{lem:riemannian-rigidity}.

\medskip
\noindent\emph{Failure of tightness at finite volume.}
The lower curvature bound may still hold when the suspension is no longer
compact.  The next proposition is a spectral criterion on the
universal cover; the example following it constructs a finite-volume
quotient satisfying the criterion.

\begin{proposition}
\label{prop:hyperbolic-horoball-criterion}
Let \((X^m,g_X)\), \(m\geq2\), be complete and simply connected, and
suppose that
\[
  \sec_{g_X}\leq-1.
\]
If \(X\) contains an open subset isometric to a horoball in
\(\mathbb H^m(-1)\), then
\begin{equation}\label{eq:cusp-spectrum}
  \lambda_1(X)=\frac{(m-1)^2}{4}.
\end{equation}
\end{proposition}

\begin{proof}
McKean's inequality gives
\[
  \lambda_1(X)\geq\frac{(m-1)^2}{4}.
\]
Let \(\mathcal H\subset X\) be an isometrically embedded hyperbolic
horoball.  For every \(R>0\), the horoball contains a relatively compact
domain isometric to the radius-\(R\) ball
\(B_R^{\mathbb H}\subset\mathbb H^m(-1)\).  Transporting compactly
supported test functions from this domain to \(X\) gives
\[
  \lambda_1(X)
  \leq\lambda_1^D(B_R^{\mathbb H}).
\]
The balls \(B_R^{\mathbb H}\) exhaust hyperbolic space, and hence
\[
  \lambda_1^D(B_R^{\mathbb H})
  \longrightarrow
  \lambda_1\bigl(\mathbb H^m(-1)\bigr)
  =\frac{(m-1)^2}{4}.
\]
This proves the reverse inequality and therefore~\eqref{eq:cusp-spectrum}.
\end{proof}

\begin{remark}
\label{rem:cusp-p-spectrum}
The same argument applies for every \(1<p<\infty\).  Replacing McKean's
inequality by its \(p\)-version and using the variational
\(p\)-spectrum of hyperbolic space gives
\[
  \lambda_{1,p}(X)
  =
  \left(\frac{m-1}{p}\right)^p.
\]
Thus this cusp construction is not specific to the linear Laplacian.  The
corresponding defect identity and rigidity theorem are developed in
Appendix~\ref{app:p-rigidity}.
\end{remark}

Proposition~\ref{prop:hyperbolic-horoball-criterion} is only a criterion:
it neither produces a finite-volume quotient nor asserts that \(X\) is
nonhyperbolic.  The following construction provides both missing features
and has a two-sided sectional-curvature bound.

\begin{example}
\label{ex:bounded-curvature-finite-volume}
We construct the promised manifold.  It is complete, noncompact, and of
finite volume, with bounded negative sectional curvature; its universal
cover attains the sharp McKean value but is not hyperbolic.

\smallskip
\noindent\emph{Construction.}
Let \((M_0^m,g_0)\) be a complete noncompact finite-volume hyperbolic
manifold of constant sectional curvature \(-4\).  Choose one cusp and a
horospherical coordinate in which, for \(t\geq T\),
\[
  g_0=dt^2+\ee^{-4t}g_N,
\]
where \((N^{m-1},g_N)\) is compact and flat.  Choose a smooth
nonincreasing function
\[
  q\colon[T,\infty)\longrightarrow[1,2]
\]
that equals \(2\) near \(T\) and equals \(1\) for all sufficiently large
\(t\).  Set
\[
  f(t)
  =\ee^{-2T}
   \exp\left(-\int_T^tq(s)\,ds\right).
\]
Since \(q=2\) near \(T\), one has \(f(t)=\ee^{-2t}\) there.  Hence
\[
  g=dt^2+f(t)^2g_N
\]
fits smoothly with \(g_0\).  Replace \(g_0\) by this metric on the chosen
cusp and leave it unchanged elsewhere.  Denote the resulting manifold by
\((M,g)\).

\smallskip
\noindent\emph{Curvature bounds.}
Since \(g_N\) is flat and
\[
  \frac{f'}{f}=-q,
  \quad
  \frac{f''}{f}=q^2-q',
\]
the warped-product curvature formulas give, for a unit vector \(Y\)
tangent to \(N\) and a two-plane \(\sigma\) tangent to \(N\),
\[
  \sec_g(\partial_t,Y)=q'-q^2,
  \quad
  \sec_g(\sigma)=-q^2.
\]
Because \(1\leq q\leq2\) and \(q'\leq0\), both curvatures are at most
\(-1\).  Moreover, \(q'\) is bounded and
\[
  q'-q^2\geq-\|q'\|_\infty-4,
  \quad
  -q^2\geq-4,
\]
while outside the transition region the curvature is either \(-4\) or
\(-1\).  Therefore, with \(K^2=4+\|q'\|_\infty\),
\[
  -K^2\leq\sec_g\leq-1
  \quad\text{on \(M\)}.
\]

\smallskip
\noindent\emph{Finite volume and sharp spectrum.}
For all sufficiently large \(t\), one has \(q(t)=1\), and hence
\[
  f(t)=C\ee^{-t}
\]
for some \(C>0\).  The inclusion of the cusp cross-section
\(N\hookrightarrow M\) is \(\pi_1\)-injective.  Indeed, before the metric
change this is a cross-section of a hyperbolic cusp of \(M_0\), and
\(\pi_1(N)\) is its parabolic cusp subgroup; changing the metric does not
alter this topological inclusion.  Consequently, every component of the
inverse image of the far end in \(X=\widetilde M\) is
\[
  [T_1,\infty)\times\widetilde N
\]
with the lifted metric
\[
  dt^2+C^2\ee^{-2t}\widetilde g_N
\]
for some sufficiently large \(T_1\).  Since \(N\) is flat,
\((\widetilde N,\widetilde g_N)\) is Euclidean; after a linear change of
the horospherical coordinates, the displayed metric is the standard
curvature-\(-1\) metric on a hyperbolic horoball.  Also,
\[
  \vol_g\bigl([T,\infty)\times N\bigr)
  =\vol_{g_N}(N)\int_T^\infty f(t)^{m-1}\,dt
  <\infty.
\]
The metric is complete because the \(t\)-direction has infinite length.
Consequently, \((M,g)\) is complete and has finite volume.

Proposition~\ref{prop:hyperbolic-horoball-criterion} gives
\[
  \lambda_1(X)=\frac{(m-1)^2}{4}.
\]
On the other hand, \(g=g_0\) on the compact core, where the sectional
curvature is \(-4\).  Hence
\[
  X\not\cong\mathbb H^m(-1).
\]
This is the asserted bounded-curvature finite-volume failure of rigidity.
\end{example}

The quotient in Example~\ref{ex:bounded-curvature-finite-volume} has
uniformly bounded sectional curvature but not bounded geometry in the
standard sense.  Indeed, the lengths of nontrivial loops in the
horospherical cross-sections of the exact cusp decay exponentially, so
\(\operatorname{inj}_g\to0\) along the cusp.  Its universal cover \(X\),
on the other hand, has infinite injectivity radius and uniformly bounded
sectional curvature.

The minimizing functions used in
Proposition~\ref{prop:hyperbolic-horoball-criterion} may be supported in
balls moving arbitrarily far into the lifted cusp.  Their projections to
\(M\), and hence the associated probability measures on
\(Z\cong SM\), escape every compact subset.  Thus the defect identities
from Section~\ref{sec:busemann-defects} still hold, but the compactness step
in Proposition~\ref{prop:defect-support} no longer produces a limiting
probability measure.  This is exactly the compactness obstruction
identified at the beginning of the subsection.

Together, the warped-product construction and
Example~\ref{ex:bounded-curvature-finite-volume} identify the two losses
caused by removing closedness.  The first shows that an upper curvature
bound alone does not force the reference horosphere to be flat.  The
second shows that even finite volume and a two-sided curvature bound do
not prevent minimizing mass from escaping.  These examples do not exclude
rigidity under additional noncompact hypotheses: a suitable tightness
condition could replace compactness of the suspension, provided that the
lower curvature control needed in the final warped-product argument
continues to hold.

\begin{problem}
\label{prob:finite-volume-entropy-rigidity}
Let \((M^m,g)\) be a complete noncompact Riemannian manifold of finite
volume, and let \(X=\widetilde M\) be its universal Riemannian cover.
Suppose that
\[
  \Ric_X\geq-(m-1)g_X
\]
and
\[
  h_{\mathrm{vol}}(X)=m-1.
\]
Is \(X\) necessarily isometric to \(\mathbb H^m(-1)\)?
\end{problem}

This asks whether the volume-entropy rigidity theorem of
Ledrappier--X.~Wang, for which Gang Liu gave a short proof in the compact
case, extends to complete finite-volume quotients (see
\cite{LedrappierWang2010}, Theorem~2, and
\cite{Liu2011}, Theorem~1).  With bounded curvature and the additional
identity
\[
  \delta_{\pi_1(M)}=h_{\mathrm{vol}}(X)=m-1,
\]
the theorem of Barthelm\'e--Marquis--Zimmer gives
\(X\cong\mathbb H^m(-1)\) (see
\cite{BarthelmeMarquisZimmer2018}, Theorem~A).  Here
\(\delta_{\pi_1(M)}\) is the Poincar\'e exponent.  Its equality with
volume entropy is automatic for cocompact actions but can fail for
finite-volume nonuniform actions, even in pinched negative curvature
(see \cite{DalboPeignePicaudSambusetti2009}, Theorem~1.1).  To the best of our
knowledge, the problem remains open.  It is distinct from the spectral
equality studied here: an exact hyperbolic cusp can realize McKean's
spectral value without forcing the maximal volume-growth condition in the
problem.

\section{Candel's leafwise semigroup construction}
\label{sec:candel-semigroup-construction}

We construct the leafwise diffusion needed for rigidity in a form valid
on general compact foliated spaces with possibly noncompact, nonproper
leaves.  Starting from the maximal continuous realization of a uniformly
elliptic leafwise Laplace operator, we obtain a conservative Feller
semigroup and identify it with the intrinsic minimal heat semigroup on
each leaf by a duplicated-leaf argument.

We also characterize stationary measures infinitesimally and prove that
their supports are leaf-saturated.

\subsection{Preparation for Candel's theorem}
\label{subsec:candel-preparation}

We use the \(C^{r,0}\) and \(C^{k,0}\) terminology fixed in
Subsection~\ref{subsec:suspension-foliated-space} and
Definition~\ref{def:Ck0}.  Let \((X,\mc F)\) be a compact metrizable
\(C^{r,0}\) foliated space of leaf dimension \(d\), where \(r\geq3\),
with locally compact metrizable transversals.  Compactness allows us to
fix a finite foliated atlas.

Consider a leafwise differential operator whose expression in a
foliated chart is
\begin{equation}\label{eq:general-leafwise-operator}
  \ms A f
  =a^{ij}(y,\tau)\partial_i\partial_jf
   +b^j(y,\tau)\partial_jf,
\end{equation}
where
\[
  a^{ij}=a^{ji},\quad
  (a^{ij})>0,\quad
  a^{ij}\in C^{2,0},\quad
  b^j\in C^{1,0}.
\]
Here and below, repeated indices are summed.  In Candel's terminology,
\(\ms A\) is a \emph{Laplace operator} on \((X,\mc F)\).  The
first-order term is allowed, but there is no zeroth-order term; in
particular, \(\ms A1=0\).  Our sign convention gives
\(\ms Af\leq0\) at every leafwise local maximum of \(f\).

The inverse matrix \((g_{ij})=(a^{ij})^{-1}\) defines the principal
leafwise metric
\[
  g_{\ms A}=g_{ij}\,dy^i\,dy^j.
\]
The compatibility of the principal symbols on chart overlaps makes this
an intrinsic metric on \(T\mc F\).  Compactness of \(X\) and a finite
foliated atlas give uniform ellipticity and uniform local coefficient
bounds.  If
\[
  \Delta_{g_{\ms A}}
  =
  a^{ij}\partial_i\partial_j+c^j\partial_j,
\]
then
\begin{equation}\label{eq:A-metric-drift-decomposition}
  \ms A=\Delta_{g_{\ms A}}+W,
  \quad
  W=(b^j-c^j)\partial_j.
\end{equation}
The vector field \(W\) has uniformly bounded leafwise length.  We also
use the fact that every leaf is complete for \(g_{\ms A}\).
To see this, suppose that a unit-speed leafwise geodesic
\(\gamma\colon[0,T)\to X\) has finite maximal time \(T\).  Compactness
gives a sequence \(t_j\uparrow T\) for which
\(\gamma(t_j)\to p\in X\).  Choose a flow box \(U\) about \(p\) and a
smaller flow box \(U'\) with \(p\in U'\) and
\(\overline{U'}\subset U\).  Uniform comparison of \(g_{\ms A}\) with
the Euclidean plaque metrics gives a number \(\delta>0\) such that
every plaquewise curve joining \(U'\) to the plaque boundary of \(U\)
has length at least \(\delta\).  For large \(j\),
\[
  \gamma(t_j)\in U',
  \quad
  T-t_j<\delta.
\]
The remaining segment of \(\gamma\) therefore stays in \(U\), and,
being leafwise, it remains in one plaque of \(U\).  Its plaque
coordinates converge as \(t\uparrow T\), and the geodesic equation
extends \(\gamma\) past \(T\), a contradiction.  Thus every leaf is
complete.  The uniform \(C^{2,0}\) bounds on the principal metric also
give a uniform lower bound on the leafwise Ricci curvature.  These
geometric consequences of compactness will enter only when
conservation of the leafwise heat-kernel mass is proved in
Subsection~\ref{subsec:candel-nonproper-heat}.

The classical continuous realization of \(\ms A\) has domain
\begin{equation}\label{eq:classical-domain}
  \mc D_0(\ms A)
  :=
  \left\{
  f\in C(X):
  f|_L\in C^2(L)\text{ for every leaf }L,\quad
  \ms Af\in C(X)
  \right\}.
\end{equation}

\begin{lemma}\label{lem:C20-density}
The space \(C^{2,0}(X)\) is uniformly dense in \(C(X)\).  In particular,
\(\mc D_0(\ms A)\) is dense in \(C(X)\).
\end{lemma}

\begin{proof}
A compact \(C^{r,0}\) foliated space admits \(C^{r,0}\) partitions of unity
subordinate to finite families of flow boxes (see
\cite{Candel2003}, Section~2).  Choose such a partition
\(\{\chi_\alpha\}\) after shrinking the flow boxes so that
\(\supp\chi_\alpha\) is compactly contained in \(U_\alpha\) and has a
uniform positive distance from the boundary in the plaque variable.
Extend
\((\chi_\alpha f)\circ\phi_\alpha^{-1}\) by zero in that variable and
convolve only in the plaque direction:
\[
  f_{\alpha,\varepsilon}(y,\tau)
  =
  \int_{\R^d}\rho_\varepsilon(y-z)
  (\chi_\alpha f)(z,\tau)\,dz.
\]
Every leafwise derivative of \(f_{\alpha,\varepsilon}\) is continuous in
\((y,\tau)\).  Uniform continuity on the compact chart closures gives
\[
  \|f_{\alpha,\varepsilon}-\chi_\alpha f\|_\infty\longrightarrow0.
\]
For sufficiently small \(\varepsilon\), each mollified term extends by
zero to a \(C^{2,0}\) function on \(X\).  The sum of the finitely many
terms converges uniformly to \(f=\sum_\alpha\chi_\alpha f\).  Hence
\(C^{2,0}(X)\) is uniformly dense in \(C(X)\).  Since
\(C^{2,0}(X)\subset\mc D_0(\ms A)\), the classical domain is dense as
well.
\end{proof}

\begin{lemma}
\label{lem:classical-gradient-control}
There is a constant \(C>0\), depending only on the foliated atlas and
\(\ms A\), such that every \(f\in\mc D_0(\ms A)\) has a continuous
leafwise gradient and
\begin{equation}\label{eq:classical-gradient-estimate}
  \sup_X|\nabla^{g_{\ms A}}f|_{g_{\ms A}}
  \leq
  C\bigl(\|f\|_\infty+\|\ms Af\|_\infty\bigr).
\end{equation}
\end{lemma}

\begin{proof}
Choose \(p>d\) and \(0<\alpha<1-d/p\).  Cover \(X\) by finitely many
smaller flow boxes \(B_0\times T_0\) compactly contained in flow boxes
\(B\times T\), with \(\overline{T_0}\) compact, and choose
\[
  B_0\Subset B_1\Subset B_2\Subset B.
\]
For \(\tau\in\overline{T_0}\), set
\[
  f_\tau(y)=f(y,\tau),
  \quad
  h_\tau(y)=(\ms Af)(y,\tau).
\]
Then \(f_\tau\in C^2(B)\) and
\(\ms A_\tau f_\tau=h_\tau\).  Uniform ellipticity and the uniform
coefficient bounds on
\(\overline{B_2}\times\overline{T_0}\) give
\[
  \|f_\tau\|_{W^{2,p}(B_1)}
  \leq
  C\bigl(\|f\|_\infty+\|\ms Af\|_\infty\bigr),
\]
with \(C\) independent of \(\tau\) (see
\cite{GilbargTrudinger2001}, Theorem~9.11, pp.~235--236).
The Sobolev--Morrey embedding gives the same bound for
\(\|f_\tau\|_{C^{1,\alpha}(\overline{B_1})}\), after decreasing
\(\alpha\) if necessary.
The equivalence of the plaque-coordinate and \(g_{\ms A}\)-norms on
the finite family of smaller boxes proves
\eqref{eq:classical-gradient-estimate}.

It remains to prove transverse continuity.  If
\(\tau_k\to\tau\) in \(T_0\), then continuity of \(f\) gives
\[
  f_{\tau_k}\longrightarrow f_\tau
  \quad\text{uniformly on }\overline{B_1}.
\]
The uniform \(C^{1,\alpha}\)-bound and the compact embedding into
\(C^1(\overline{B_0})\) show that every subsequence has a further
subsequence converging in \(C^1(\overline{B_0})\).  Its limit must be
\(f_\tau\), so the full sequence converges.  Thus the first leafwise
derivatives of \(f\) are continuous in \((y,\tau)\).  Since the
coefficients of \(g_{\ms A}\) are continuous, so is
\(\nabla^{g_{\ms A}}f\).
\end{proof}

Let \(\ms A^*\) denote the formal adjoint on a leaf relative to
\(dV_{g_{\ms A}}\).

\begin{definition}
\label{def:maximal-realization}
A function \(f\in C(X)\) belongs to
\(\mc D(\ms A_{\max})\) if there is an \(h\in C(X)\) such that, on every
leaf \(L\),
\begin{equation}\label{eq:maximal-distributional-identity}
  \int_L f\,\ms A^*\varphi\,dV_{g_{\ms A}}
  =
  \int_L h\varphi\,dV_{g_{\ms A}}
  \quad
  \bigl(\varphi\in C_c^\infty(L)\bigr).
\end{equation}
We then set \(\ms A_{\max}f=h\).
\end{definition}

The function \(h\) is unique, because a continuous function whose
distributional pairing with every compactly supported test function
vanishes must vanish on every leaf.  Leafwise integration by parts gives
\[
  \mc D_0(\ms A)\subset\mc D(\ms A_{\max}),
  \quad
  \ms A_{\max}f=\ms Af
  \quad\bigl(f\in\mc D_0(\ms A)\bigr).
\]

\begin{lemma}\label{lem:maximal-closed}
The operator
\[
  \ms A_{\max}\colon
  \mc D(\ms A_{\max})\subset C(X)\longrightarrow C(X)
\]
is closed and densely defined.
\end{lemma}

\begin{proof}
Density follows from Lemma~\ref{lem:C20-density}.  To prove closedness,
suppose that \(f_n\to f\) and \(\ms A_{\max}f_n\to h\) uniformly on
\(X\).  For every leaf \(L\) and every
\(\varphi\in C_c^\infty(L)\), the defining identity gives
\[
  \int_L f_n\,\ms A^*\varphi\,dV_{g_{\ms A}}
  =
  \int_L(\ms A_{\max}f_n)\varphi\,dV_{g_{\ms A}}.
\]
The support of \(\varphi\) is compact.  Passing to the limit proves
\eqref{eq:maximal-distributional-identity} for \(f\) and \(h\).
Therefore \(f\in\mc D(\ms A_{\max})\) and
\(\ms A_{\max}f=h\).
\end{proof}

The following theorem is the main analytic result of this section.  It
is the form of Candel's construction needed here and combines his
principal results (see \cite{Candel2003},
Propositions~4.7--4.13, Theorem~4.14, and
Propositions~4.15--4.16).

\begin{theorem}\label{thm:candel-diffusion}
Let \((X,\mc F)\) and \(\ms A\) satisfy the hypotheses above.
\begin{enumerate}[label=\textup{(\roman*)}]
\item For every \(\lambda>0\), the map
\[
  \lambda I-\ms A_{\max}\colon
  \mc D(\ms A_{\max})\longrightarrow C(X)
\]
is bijective.  Its resolvent
\[
  R_\lambda=(\lambda I-\ms A_{\max})^{-1}
\]
is positive and satisfies
\begin{equation}\label{eq:resolvent-properties}
  \|R_\lambda\|_{C(X)\to C(X)}\leq\frac1\lambda,
  \quad
  R_\lambda1=\frac1\lambda1.
\end{equation}
If \(q\in C^{2,0}(X)\), then
\(R_\lambda q\in\mc D_0(\ms A)\).

\item The operator \(\ms A_{\max}\) generates a unique strongly
continuous semigroup \(\{P_t\}_{t\geq0}\) on \(C(X)\) such that
\begin{equation}\label{eq:Markov-properties}
  P_t\geq0,\quad
  \|P_t\|\leq1,\quad
  P_t1=1.
\end{equation}
Thus \(\{P_t\}_{t\geq0}\) is a conservative Feller Markov semigroup.

\item Suppose in addition that the coefficients of \(\ms A|_L\) are
smooth on every leaf \(L\).  If \(x\in L\), then
\begin{equation}\label{eq:leafwise-kernel-formula}
  (P_tf)\bigl(\iota_L(x)\bigr)
  =
  \int_L
  k_L(t,x,y)f\bigl(\iota_L(y)\bigr)\,
  dV_{g_{\ms A}}(y)
  \quad
  \bigl(t>0,\ f\in C(X)\bigr),
\end{equation}
where \(k_L\) is the smooth minimal heat kernel of
\(\partial_t-\ms A|_L\).  Moreover,
\begin{equation}\label{eq:kernel-positive-conservative}
  k_L(t,x,y)>0,
  \quad
  \int_L k_L(t,x,y)\,dV_{g_{\ms A}}(y)=1.
\end{equation}
Equivalently, the transition probability from \(\iota_L(x)\) is the
pushforward under the leaf inclusion \(\iota_L\colon L\to X\) of
\[
  k_L(t,x,y)\,dV_{g_{\ms A}}(y).
\]
\end{enumerate}
\end{theorem}

\medskip
\noindent\textbf{Roadmap and analytic difficulty.}
The remaining subsections separate standard analytic tools from the
two points that are specific to the foliated setting.
\begin{itemize}[leftmargin=2em]
\item In Subsection~\ref{subsec:candel-resolvent-feller} we prove
parts~\textup{(i)}--\textup{(ii)}.  The resolvent estimate and the
Hille--Yosida argument are standard.  The first essential point is
transverse continuity: leafwise Dirichlet exhaustions initially give
only lower and upper semicontinuous solutions, and Candel's argument
uses compactness of \(X\) to show that they agree.

\item In Subsection~\ref{subsec:candel-nonproper-heat} we prove
part~\textup{(iii)}.  Standard heat-kernel theory on a smooth
Riemannian manifold is taken as background.  The main difficulty is
to identify the global semigroup with intrinsic heat diffusion on a
nonproper leaf; the duplicated-leaf construction and its intertwining
relations resolve this issue.

\item In Subsection~\ref{subsec:candel-core} we prove that the
classical leafwise domain is a core for the generator.  Once the
resolvent regularity is known, this is a standard graph-norm
approximation.
\end{itemize}

\medskip
\noindent\textbf{Relation with Candel's original proof.}
Theorem~\ref{thm:candel-diffusion} is due to Candel; its construction
is developed in detail in his paper (see \cite{Candel2003},
Propositions~4.7--4.16).  The proof
below follows the same structure but gives the additional details
needed in the present finite-regularity setting.  For
parts~\textup{(i)}--\textup{(ii)}, we describe the maximal
distributional realization, establish the resolvent estimates, and
apply the Hille--Yosida theorem.  For part~\textup{(iii)}, we expand
Candel's duplicated-leaf argument from Propositions~4.15--4.16: we
verify the topology of the enlarged foliated space, define the two
natural embeddings of function spaces, prove the resolvent and
semigroup intertwining identities, and identify the induced
\(C_0\)-semigroup with the intrinsic minimal heat semigroup.  We then
prove conservation of the leafwise heat-kernel mass directly from
completeness, a Ricci lower bound, bounded drift, and the parabolic
maximum principle.  The same proof applies to other leafwise elliptic
operators on compact foliated spaces with noncompact, nonproper leaves.

\subsection{The resolvent and the Feller semigroup}
\label{subsec:candel-resolvent-feller}

We use the Hille--Yosida theorem in the following form
(see \cite{Brezis2011}, Theorems~7.8--7.9).

\begin{hilleyosida}
\leavevmode\par\smallskip
Let \(E\) be a Banach space and let
\(\ms B\colon\mc D(\ms B)\subset E\to E\) be a densely defined linear
operator.  Suppose that, for every \(\mu>0\), the operator
\[
  I+\mu\ms B\colon\mc D(\ms B)\longrightarrow E
\]
is bijective and
\[
  \bigl\|(I+\mu\ms B)^{-1}\bigr\|_{E\to E}\leq1.
\]
Then \(-\ms B\) generates a unique strongly continuous contraction
semigroup \(\{P_t\}_{t\geq0}\) on \(E\).  Moreover,
\[
  P_tf
  =
  \lim_{n\to\infty}
  \left(I+\frac tn\ms B\right)^{-n}f
  \quad(t>0,\ f\in E),
\]
where the limit is taken in \(E\).
\end{hilleyosida}

We first construct the resolvent of \(\ms A_{\max}\).  Once surjectivity
and the norm estimate are established, the Feller semigroup follows
directly from the theorem above.

\begin{proof}[Proof of parts~\textup{(i)} and~\textup{(ii)} of
Theorem~\ref{thm:candel-diffusion}]
\leavevmode\par\smallskip

\noindent\emph{1. The resolvent estimate.}
Fix \(\lambda>0\) and \(u\in\mc D_0(\ms A)\).  Choose \(x_0\in X\) such
that \(|u(x_0)|=\|u\|_\infty\).  If
\(u(x_0)=\|u\|_\infty\), then \(x_0\) is a local maximum of \(u\) on its
leaf, so \(\ms Au(x_0)\leq0\).  Hence
\[
  \bigl|(\lambda I-\ms A)u(x_0)\bigr|
  =
  \lambda\|u\|_\infty-\ms Au(x_0)
  \geq\lambda\|u\|_\infty.
\]
If \(u(x_0)=-\|u\|_\infty\), apply the same argument to \(-u\).
Therefore
\begin{equation}\label{eq:classical-resolvent-estimate}
  \|\lambda u-\ms Au\|_\infty
  \geq\lambda\|u\|_\infty.
\end{equation}
In particular, a classical solution of the resolvent equation is unique,
and any such solution satisfies
\[
  \|u\|_\infty\leq\lambda^{-1}
  \|(\lambda I-\ms A)u\|_\infty.
\]

\smallskip
\noindent\emph{2. Existence for \(C^{2,0}\) data.}
Let \(q\in C^{2,0}(X)\), and set
\[
  c_-=-\frac{\|q\|_\infty}{\lambda},
  \quad
  c_+=\frac{\|q\|_\infty}{\lambda}.
\]
On every relatively compact regular domain \(D\) in a leaf, solve
\[
  \lambda u_D^\pm-\ms Au_D^\pm=q\quad\text{in }D,
  \quad
  u_D^\pm=c_\pm\quad\text{on }\partial D.
\]
Since
\[
  (\lambda I-\ms A)c_-=-\|q\|_\infty\leq q
  \leq\|q\|_\infty=(\lambda I-\ms A)c_+,
\]
comparison gives \(c_-\leq u_D^\pm\leq c_+\).  If \(D\Subset E\), apply
comparison on \(D\) to \(u_D^\pm\) and \(u_E^\pm|_D\).  This gives
\[
  c_-\leq u_D^-\leq u_E^-\leq u_E^+\leq u_D^+\leq c_+
  \quad(D\Subset E).
\]
Thus the lower solutions increase and the upper solutions decrease
under exhaustion.  For \(x\in X\), define
\[
  u_-(x)=\sup_{\substack{D\Subset L(x)\\x\in D}}u_D^-(x),
  \quad
  u_+(x)=\inf_{\substack{D\Subset L(x)\\x\in D}}u_D^+(x).
\]
Interior Schauder estimates and a diagonal argument show that these
bounded functions are \(C^2\) on every leaf and satisfy
\[
  \lambda u_\pm-\ms Au_\pm=q,
  \quad
  u_-\leq u_+.
\]
Comparison also shows that \(u_-\) is the smallest bounded solution
that is everywhere at least \(c_-\), whereas \(u_+\) is the largest
bounded solution that is everywhere at most \(c_+\).

We next verify the required semicontinuity across leaves.  Fix
\(x\in X\), choose a lift \(\widehat x\) to the holonomy cover
\(\widehat L\to L(x)\), and let \(\varepsilon>0\).  The extremal
construction is unchanged by passage to this cover.  Indeed, the lower
extremal solution on \(\widehat L\) is invariant under every deck
transformation, because its defining family of Dirichlet problems is
deck-invariant.  It therefore descends to a bounded solution on
\(L(x)\) that is at least \(c_-\).  Minimality of the lower extremal
solution on \(L(x)\) gives one inequality, while minimality on
\(\widehat L\), applied to the lift of \(u_-\), gives the reverse
inequality.  Thus the two solutions agree under pullback.  The same
argument, using maximality, applies to the upper solutions.

It follows that there is a bounded regular domain
\(D\Subset\widehat L\), containing \(\widehat x\), for which the lower
Dirichlet solution \(v_-\) satisfies
\[
  v_-(\widehat x)>u_-(x)-\varepsilon.
\]
Because the holonomy cover has trivial holonomy, a neighborhood of
\(\overline{D}\) has a foliated thickening.  More precisely, after
shrinking around \(\overline{D}\), there are a transversal \(T_0\), a
distinguished point \(\tau_0\in T_0\), and a leafwise local
diffeomorphism
\[
  \Psi\colon \overline{D}\times T_0\longrightarrow X
\]
such that \(\Psi(\,\cdot\,,\tau_0)\) is the covering map on
\(\overline{D}\).  Pull back \(\ms A\) and \(q\) by \(\Psi\).  On each
copy \(D\times\{\tau\}\), let \(v_\tau\) solve the Dirichlet problem
with boundary value \(c_-\); thus \(v_{\tau_0}=v_-\).

The coefficients, the right-hand side, and the induced boundary
geometry vary continuously with \(\tau\) in the norms required by the
Schauder estimates.  Hence the family \(\{v_\tau\}\) is locally
precompact in \(C^2(\overline{D})\).  If \(\tau_j\to\tau\), every
convergent subsequence solves the Dirichlet problem at \(\tau\);
uniqueness identifies its limit with \(v_\tau\).  The whole family
therefore depends continuously on \(\tau\).  Comparison on
\(D\times\{\tau\}\) gives
\[
  v_\tau\leq u_-\circ\Psi
\]
wherever this expression is defined.  Applying this inequality to
points converging to \((\widehat x,\tau_0)\) yields
\[
  \liminf_{y\to x}u_-(y)
  \geq v_-(\widehat x)>u_-(x)-\varepsilon.
\]
Thus \(u_-\) is lower semicontinuous.  The same argument, with boundary
value \(c_+\) and the inequalities reversed, shows that \(u_+\) is
upper semicontinuous.  This is Candel's transverse argument (see
\cite{Candel2003}, Lemma~4.8 and Proposition~4.9).

Consequently,
\[
  h:=u_+-u_-
\]
is a bounded nonnegative upper semicontinuous function, is \(C^2\) on
every leaf, and satisfies
\[
  \ms Ah=\lambda h.
\]
Compactness of \(X\) gives a point where \(h\) attains its maximum.  If
this maximum were positive, the leafwise maximum principle would give
\(\ms Ah\leq0\) there, contrary to \(\ms Ah=\lambda h>0\).  Therefore
\(u_-=u_+\).  Their common value \(u\) is both lower and upper
semicontinuous, hence continuous, and
\begin{equation}\label{eq:classical-resolvent-equation}
  \lambda u-\ms Au=q,
  \quad
  \|u\|_\infty\leq\lambda^{-1}\|q\|_\infty.
\end{equation}
Since \(\ms Au=\lambda u-q\) is continuous on \(X\), we have
\(u\in\mc D_0(\ms A)\).  This proves the resolvent statement for
\(C^{2,0}\) data (see \cite{Candel2003}, Proposition~4.10).

\smallskip
\noindent\emph{3. Continuous data and the maximal realization.}
For an arbitrary \(q\in C(X)\), choose
\(q_j\in C^{2,0}(X)\) converging uniformly to \(q\), and let
\(u_j\in\mc D_0(\ms A)\) solve
\[
  \lambda u_j-\ms Au_j=q_j.
\]
The estimate \eqref{eq:classical-resolvent-estimate} gives
\[
  \|u_j-u_k\|_\infty
  \leq\lambda^{-1}\|q_j-q_k\|_\infty.
\]
Thus \(u_j\to u\) uniformly for some \(u\in C(X)\), while
\[
  \ms Au_j=\lambda u_j-q_j
  \longrightarrow\lambda u-q
\]
uniformly.  Lemma~\ref{lem:maximal-closed} gives
\[
  u\in\mc D(\ms A_{\max}),
  \quad
  (\lambda I-\ms A_{\max})u=q.
\]
Thus \(\lambda I-\ms A_{\max}\) is onto \(C(X)\).

If \(f\in\mc D(\ms A_{\max})\) satisfies
\((\lambda I-\ms A_{\max})f=0\), then, on each plaque,
\(\ms Af=\lambda f\) in the distributional sense.  Interior elliptic
regularity first gives \(f\in W_{\mathrm{loc}}^{2,p}\) for every finite
\(p\).  Taking \(p>d\) and then applying local Schauder regularity
shows that \(f\) is \(C^2\) on each plaque.  A positive global maximum
of \(f\), or a negative global minimum, now contradicts the pointwise
maximum principle for
\(\lambda f-\ms Af=0\).  Compactness of \(X\) therefore implies
\(f=0\).  Hence the resolvent equation is injective as well as
surjective.  This is Candel's argument (see \cite{Candel2003},
Propositions~4.12--4.13), with the domain written as
in Definition~\ref{def:maximal-realization}.

We may now set
\[
  R_\lambda=(\lambda I-\ms A_{\max})^{-1}.
\]
The estimate for the approximating solutions passes to the uniform limit:
\begin{equation}\label{eq:maximal-resolvent-estimate}
  \|R_\lambda q\|_\infty
  \leq\lambda^{-1}\|q\|_\infty.
\end{equation}
If \(q\geq0\), choose \(\varepsilon_j\downarrow0\) and
\(p_j\in C^{2,0}(X)\) such that
\(\|p_j-q\|_\infty<\varepsilon_j\), and set
\(q_j=p_j+\varepsilon_j\).  Then \(q_j\geq0\) and \(q_j\to q\)
uniformly.  The minimum principle applied to the classical solutions
\(R_\lambda q_j\) gives \(R_\lambda q_j\geq0\).  Passing to the uniform
limit shows that \(R_\lambda q\geq0\).  Thus \(R_\lambda\) is positive.
Finally,
\[
  (\lambda I-\ms A_{\max})(\lambda^{-1}1)=1,
\]
so uniqueness gives \(R_\lambda1=\lambda^{-1}1\).  This proves
part~\textup{(i)}.

\smallskip
\noindent\emph{4. Verification of the Hille--Yosida hypotheses.}
Put
\[
  E=C(X),
  \quad
  \ms B=-\ms A_{\max},
  \quad
  \mc D(\ms B)=\mc D(\ms A_{\max}).
\]
The domain of \(\ms B\) is dense by Lemma~\ref{lem:maximal-closed}.  For
every \(\mu>0\),
\[
  I+\mu\ms B
  =
  I-\mu\ms A_{\max}
  =
  \mu\bigl(\mu^{-1}I-\ms A_{\max}\bigr)
\]
is bijective, and
\begin{equation}\label{eq:accretive-resolvent}
  (I+\mu\ms B)^{-1}
  =
  \mu^{-1}R_{\mu^{-1}},
  \quad
  \|(I+\mu\ms B)^{-1}\|\leq1
\end{equation}
by \eqref{eq:maximal-resolvent-estimate}.  Thus \(\ms B\) is
\(m\)-accretive in the standard terminology (see
\cite{Brezis2011}, Theorem~7.8).  The Hille--Yosida theorem
therefore gives a unique strongly continuous contraction semigroup
\(\{P_t\}_{t\geq0}\) on \(E\) solving
\[
  \frac{d}{dt}P_tf+\ms B P_tf=0
  \quad
  \bigl(f\in\mc D(\ms B)\bigr).
\]
Equivalently, the generator of this semigroup is
\(-\ms B=\ms A_{\max}\).

The Hille--Yosida exponential formula gives
\begin{equation}\label{eq:Yosida-formula}
  P_tf
  =
  \lim_{n\to\infty}
  \left(I-\frac tn\ms A_{\max}\right)^{-n}f
  =
  \lim_{n\to\infty}
  \left(\frac ntR_{n/t}\right)^nf
  \quad(t>0)
\end{equation}
with convergence in \(C(X)\).  By
\eqref{eq:accretive-resolvent}, each factor
\((n/t)R_{n/t}\) is a contraction.  Part~\textup{(i)} also gives
\[
  \frac ntR_{n/t}\geq0,
  \quad
  \frac ntR_{n/t}1=1.
\]
Passing to the limit in \eqref{eq:Yosida-formula} shows that every
\(P_t\) is positive, contractive, and fixes \(1\).  Therefore
\(\{P_t\}_{t\geq0}\) is a conservative Markov semigroup.  Since it is a
strongly continuous semigroup on \(C(X)\), it is Feller.  This proves
part~\textup{(ii)}.
\end{proof}

\subsection{The heat kernel on a nonproper leaf}
\label{subsec:candel-nonproper-heat}

We prove Theorem~\ref{thm:candel-diffusion}\textup{(iii)} by identifying
the global semigroup on \(C(X)\) with the intrinsic heat semigroup on a
possibly nonproper leaf.

We use standard heat-kernel theory, local elliptic and parabolic
regularity, and the strong maximum principle.  A smooth elliptic
Laplace operator without a zeroth-order term has a smooth minimal
nonnegative heat kernel obtained by Dirichlet exhaustion, and its heat
operators form a strongly continuous positive contraction semigroup
on \(C_0\).  For the Laplace--Beltrami operator this is standard
(see \cite{Li2012}, Chapters~10--11).  The same arguments apply to the
smooth drift operator \(\ms A_L=\Delta_{g_{\ms A,L}}+W\).
It remains to prove that the leafwise heat semigroup agrees with the
global semigroup.

\begin{proof}[Proof of part~\textup{(iii)} of
Theorem~\ref{thm:candel-diffusion}]

Fix a leaf \(L\), let
\(\iota_L\colon L\to X\) be its inclusion, and fix \(x\in L\).  The
argument has six steps.

\smallskip
\noindent\textbf{Step 1: construction of the duplicated-leaf space.}
The difficulty is that the intrinsic topology of a leaf
may be strictly finer than its subspace topology in \(X\).  Thus a
function in \(C_0(L)\) need not extend continuously to \(X\), and one
cannot define a semigroup on \(C_0(L)\) by choosing an extension and
restricting \(P_t\).  Indeed, the result would be independent of the
extension only if the transition probability from a point of \(L\)
were already known to be concentrated on \(L\), which is the conclusion
we are proving.

We use Candel's duplicated-leaf construction (see
\cite{Candel2003}, Proposition~4.15).  Let \(\bar{L}\) be a second copy
of \(L\), and let
\[
  s\colon\bar{L}\longrightarrow L
\]
be the canonical diffeomorphism.  Set
\[
  \bar{X}_L:=X\sqcup\bar{L}
\]
as a disjoint union of sets.

\smallskip
\noindent\emph{The topology.}
We define a topology on \(\bar{X}_L\) by declaring the following subsets to
form a basis:
\begin{equation}\label{eq:duplicated-leaf-basis}
  s^{-1}(V),
  \quad
  \mc N(U,K)
  :=
  U\cup s^{-1}\bigl(\iota_L^{-1}(U)\setminus K\bigr),
\end{equation}
where \(V\subset L\) is intrinsically open, \(U\subset X\) is open, and
\(K\subset L\) is compact.  For example,
\[
  s^{-1}(V)\cap\mc N(U,K)
  =
  s^{-1}\bigl(V\cap\iota_L^{-1}(U)\setminus K\bigr)
\]
and
\[
  \mc N(U_1,K_1)\cap\mc N(U_2,K_2)
  =
  \mc N(U_1\cap U_2,K_1\cup K_2),
\]
which verifies the basis intersection condition.

The induced topology on \(X\subset \bar{X}_L\) is the original topology of
\(X\), while the induced topology on \(\bar{L}\subset \bar{X}_L\) is its
intrinsic manifold topology.

\smallskip
\noindent\emph{Topological properties.}
The required properties may be checked separately.
\begin{itemize}[leftmargin=2em]
\item \emph{Compactness.}
Let an open cover of \(\bar{X}_L\) be given.  Since \(X\) is compact, finitely
many basic sets \(\mc N(U_i,K_i)\) cover \(X\).  If
\(K=K_1\cup\cdots\cup K_N\), these same sets cover
\[
  X\cup s^{-1}(L\setminus K).
\]
The remaining set \(s^{-1}(K)\) is compact in \(\bar{L}\), so it has a
finite subcover.  Hence \(\bar{X}_L\) is compact.

\item \emph{Hausdorff property.}
For two distinct points of \(X\), choose disjoint open neighborhoods
\(U_1,U_2\subset X\).  Then
\(\mc N(U_1,\varnothing)\) and \(\mc N(U_2,\varnothing)\) separate
them in \(\bar{X}_L\).  Two points of \(\bar{L}\) are separated by disjoint
intrinsic open subsets of \(\bar{L}\).  If \(z\in X\) and
\(\bar{y}\in\bar{L}\), choose a relatively compact intrinsic neighborhood
\(V\) of \(s(\bar{y})\).  Then \(s^{-1}(V)\) and
\(\mc N(X,\overline V)\) are disjoint neighborhoods of \(\bar{y}\) and
\(z\), respectively.  Thus \(\bar{X}_L\) is Hausdorff.

\item \emph{Second countability and metrizability.}
Let \(\{U_i\}\) and \(\{V_j\}\) be countable bases of \(X\) and \(L\),
respectively, and choose a compact exhaustion
\[
  K_1\subset\operatorname{int}K_2\subset K_2
  \subset\operatorname{int}K_3\subset\cdots,
  \quad
  \bigcup_nK_n=L.
 \]
The sets
\[
  s^{-1}(V_j),
  \quad
  \mc N(U_i,K_n)
\]
form a countable basis for \(\bar{X}_L\).  Since \(\bar{X}_L\) is
compact, Hausdorff, and second countable, it is metrizable.

\item \emph{Foliated structure.}
Candel's duplicated-plaque construction equips \(\bar{X}_L\) with a foliated
structure.  Its original leaves are the leaves of \(X\), while \(\bar{L}\),
with its intrinsic manifold structure, is one additional leaf.
Concretely, shrink a flow box \(U\cong B\times T\) in \(X\) and choose
the compact set \(K\) so that the components of
\(\iota_L^{-1}(U)\setminus K\) meeting the smaller flow box are
plaques.  Adjoin copies of these plaques to the original flow box and
topologize their transverse parameters by
\eqref{eq:duplicated-leaf-basis}.  This gives another local product
\(B\times T'\).  Its plaque coordinate changes are the original
coordinate changes on \(X\); hence the resulting atlas is \(C^{r,0}\).
\end{itemize}

Because the space is metrizable, its topology may be tested by
sequences.  Directly from \eqref{eq:duplicated-leaf-basis}, a sequence
\(\bar{y}_n\in\bar{L}\) converges to \(z\in X\) if and only if
\begin{equation}\label{eq:duplicated-leaf-convergence}
  \iota_L(s(\bar{y}_n))\longrightarrow z\quad\text{in }X,
  \quad
  s(\bar{y}_n)\longrightarrow\infty\quad\text{in }L.
\end{equation}
Here convergence to infinity means that the sequence is eventually
outside every compact subset of \(L\).

\smallskip
\noindent\textbf{Step 2: function spaces and the closed leaf subspace.}
The copy \(\bar L\) has the intrinsic topology of \(L\) and is
topologically embedded as an open leaf of \(\bar X_L\).  The
construction therefore changes no leafwise geometry while making zero
extension available: \(C_0(\bar L)\) becomes a closed subspace of
\(C(\bar X_L)\).  This is precisely what fails when the intrinsic
topology of \(L\) is finer than its subspace topology in \(X\), and it
is the functional-analytic reason for passing to the duplicated space.

We regard \(X\) and \(\bar{L}\) as the two disjoint summands of \(\bar{X}_L\).
For \(F\in C(\bar{X}_L)\), the symbols \(F|_X\) and \(F|_{\bar{L}}\) denote
the corresponding restrictions.  The space \(C_0(\bar{L})\) consists of
the continuous functions on the intrinsic manifold \(\bar{L}\) that
vanish at infinity; equivalently, \(\varphi\in C_0(\bar L)\) if and
only if
\[
  \{\bar{y}\in\bar{L}:|\varphi(\bar{y})|\geq\varepsilon\}
  \quad\text{is compact for every }\varepsilon>0.
\]

\smallskip
\noindent\emph{The two embeddings.}
Define

\begin{equation}\label{eq:duplicated-leaf-embeddings}
  \begin{aligned}
  &J\colon C(X)\longrightarrow C(\bar{X}_L),
  &&Jf|_X=f,
  &&Jf(\bar{y})=f\bigl(\iota_L(s(\bar{y}))\bigr),\\
  &I\colon C_0(\bar{L})\longrightarrow C(\bar{X}_L),
  &&I\varphi|_X=0,
  &&I\varphi|_{\bar{L}}=\varphi.
  \end{aligned}
\end{equation}
Thus \(J\) copies \(f\) to the new leaf through
\(\iota_L\circ s\), whereas \(I\) extends a function on \(\bar{L}\) by
zero on \(X\).  Both maps are linear isometries:
\[
  \|Jf\|_{C(\bar{X}_L)}=\|f\|_{C(X)},
  \quad
  \|I\varphi\|_{C(\bar{X}_L)}=\|\varphi\|_{C_0(\bar{L})}.
\]
Continuity of \(Jf\) at a point of the original space \(X\) follows from
\eqref{eq:duplicated-leaf-convergence}.  If
\(\varphi\in C_0(\bar{L})\) and \(\varepsilon>0\), choose a compact
\(K\subset L\) such that
\(|\varphi(\bar{y})|<\varepsilon\) whenever \(s(\bar{y})\notin K\).  The zero
extension \(I\varphi\) then has absolute value less than
\(\varepsilon\) on \(\mc N(X,K)\cap\bar{L}\), which proves its continuity
at every point of \(X\).

\smallskip
\noindent\emph{The closed subspace associated with the copied leaf.}
Since \(I\) is an isometry, its range is closed.  We claim that
\[
  Y:=IC_0(\bar{L})
  =
  \{F\in C(\bar{X}_L):F|_X=0\}.
\]
Only the reverse inclusion needs verification.  If \(F|_X=0\) but
\(F|_{\bar{L}}\) did not vanish at infinity, there would be
\(\varepsilon>0\) and a sequence \(\bar{y}_n\) escaping every compact
subset of \(\bar{L}\) such that
\(|F(\bar{y}_n)|\geq\varepsilon\).  After taking a
subsequence, compactness of \(X\) gives
\(\iota_L(s(\bar{y}_n))\to z\in X\).  Equation
\eqref{eq:duplicated-leaf-convergence} then gives \(\bar{y}_n\to z\) in
\(\bar{X}_L\), contradicting \(F(\bar{y}_n)\to F(z)=0\).
Thus the needed closed subspace exists in \(C(\bar X_L)\), although
the analogous zero-extension subspace generally does not exist in
\(C(X)\).

\smallskip
\noindent\textbf{Step 3: the copied metric, operator, and semigroup.}
Write \(g_{\ms A,L}\) and \(\ms A_L\) for the restrictions to \(L\) of
the principal metric and the leafwise operator, respectively.  Define
their copies on \(\bar{L}\) by
\[
  g_{\ms A,\bar{L}}=s^*g_{\ms A,L},
  \quad
  \ms A_{\bar{L}}(\varphi\circ s)
  =(\ms A_L\varphi)\circ s
  \quad\bigl(\varphi\in C^\infty(L)\bigr).
\]
Then \(s\) is an isometry and
\begin{equation}\label{eq:copied-volume-form}
  s^*dV_{g_{\ms A,L}}=dV_{g_{\ms A,\bar{L}}}.
\end{equation}
On leafwise smooth functions \(F\) on \(\bar{X}_L\), define
\(\bar{\ms A}\) by
\[
  (\bar{\ms A}F)|_X=\ms A(F|_X),
  \quad
  (\bar{\ms A}F)|_{\bar{L}}
  =\ms A_{\bar{L}}(F|_{\bar{L}}).
\]
The copied local formulas have the same transverse regularity as the
original coefficients.  Hence parts~\textup{(i)}--\textup{(ii)}, applied
to \((\bar{X}_L,\bar{\ms A})\), give
\[
  \bar{R}_\lambda
  :=
  (\lambda I-\bar{\ms A}_{\max})^{-1}
  \quad(\lambda>0)
\]
and a conservative Feller semigroup
\(\{\bar{P}_t\}_{t\geq0}\) on \(C(\bar{X}_L)\) with generator
\(\bar{\ms A}_{\max}\).  Here \(\bar{\ms A}_{\max}\) is the maximal
continuous realization on \(C(\bar{X}_L)\) defined as in
Definition~\ref{def:maximal-realization}.

Finally, \(\ms A_{\bar{L},\max}\) denotes the maximal distributional
realization of \(\ms A_{\bar{L}}\) on \(C_0(\bar{L})\).  Explicitly,
\(v\in\mc D(\ms A_{\bar{L},\max})\) if there is an
\(h\in C_0(\bar{L})\) such that
\[
  \int_{\bar{L}}v\,\ms A_{\bar{L}}^*\psi\,
  dV_{g_{\ms A,\bar{L}}}
  =
  \int_{\bar{L}}h\psi\,dV_{g_{\ms A,\bar{L}}}
  \quad
  \bigl(\psi\in C_c^\infty(\bar{L})\bigr),
\]
where \(\ms A_{\bar{L}}^*\) is the formal adjoint relative to
\(dV_{g_{\ms A,\bar{L}}}\).  We then set
\(\ms A_{\bar{L},\max}v=h\).  The proof of
Lemma~\ref{lem:maximal-closed} applies verbatim on the locally compact
manifold \(\bar L\): the operator \(\ms A_{\bar L,\max}\) is closed,
and it is densely defined because
\(C_c^\infty(\bar L)\subset\mc D(\ms A_{\bar L,\max})\) is dense in
\(C_0(\bar L)\).

\smallskip
\noindent\textbf{Step 4: the two resolvent intertwining identities.}
We show that the resolvent on the duplicated space respects both
canonical embeddings: \(J\) identifies the original global problem
with its copy on \(\bar{X}_L\), while \(I\) identifies the intrinsic
problem on \(\bar{L}\) with the closed subspace \(Y\).

\smallskip
\noindent\emph{The \(J\)-intertwining.}
Let \(f\in C(X)\), and put \(u=R_\lambda f\).  On every original leaf,
\[
  (\lambda I-\bar{\ms A}_{\max})Ju=Jf.
\]
The same identity holds distributionally on \(\bar{L}\), because both the
operator and the restrictions of \(u\) and \(f\) have been copied
through \(s\).  Thus \(Ju\in\mc D(\bar{\ms A}_{\max})\) and the
identity holds on all of \(\bar{X}_L\).  Uniqueness of the resolvent
equation gives
\begin{equation}\label{eq:J-resolvent-intertwining}
  \bar{R}_\lambda J=JR_\lambda.
\end{equation}

\smallskip
\noindent\emph{The \(I\)-intertwining.}
Take \(\varphi\in C_0(\bar{L})\) and set
\[
  F=\bar{R}_\lambda I\varphi.
\]
Restriction of the resolvent equation to the original space gives
\[
  (\lambda I-\ms A_{\max})(F|_X)=0.
\]
Here \(F|_X\in\mc D(\ms A_{\max})\), because the distributional
identity defining \(\bar{\ms A}_{\max}\) restricts to every original
leaf.
The injectivity proved in part~\textup{(i)} implies \(F|_X=0\).
Therefore \(F\in Y\), so \(F=Iv\) for a unique \(v\in C_0(\bar{L})\).
Restricting the distributional equation to \(\bar{L}\) gives
\begin{equation}\label{eq:copied-leaf-resolvent-equation}
  (\lambda I-\ms A_{\bar{L},\max})v=\varphi.
\end{equation}
This proves surjectivity of
\(\lambda I-\ms A_{\bar L,\max}\).  If
\((\lambda I-\ms A_{\bar L,\max})v=0\), then \(Iv\) solves the
homogeneous resolvent equation on \(\bar X_L\), so \(v=0\).
Thus \(\lambda I-\ms A_{\bar{L},\max}\) is bijective on
\(C_0(\bar{L})\), and, with
\[
  R_\lambda^{\bar{L}}
  :=
  (\lambda I-\ms A_{\bar{L},\max})^{-1},
\]
we have
\begin{equation}\label{eq:I-resolvent-intertwining}
  \bar{R}_\lambda I=IR_\lambda^{\bar{L}}.
\end{equation}

The identities \eqref{eq:J-resolvent-intertwining} and
\eqref{eq:I-resolvent-intertwining} show that the following
diagram commutes:
\[
\begin{tikzcd}[column sep=large,row sep=large]
C(X)
  \arrow[r,"J",color=blue!70!black]
  \arrow[d,"R_\lambda"',color=blue!70!black]
&
C(\bar{X}_L)
  \arrow[d,"\bar{R}_\lambda"]
&
C_0(\bar{L})
  \arrow[l,"I"',color=red!70!black]
  \arrow[d,"R_\lambda^{\bar{L}}",color=red!70!black]
\\
C(X)
  \arrow[r,"J"',color=blue!70!black]
&
C(\bar{X}_L)
&
C_0(\bar{L})
  \arrow[l,"I",color=red!70!black]
\end{tikzcd}
\]
The left blue square is the \(J\)-intertwining, and the right red
square is the \(I\)-intertwining.

\smallskip
\noindent\emph{Passage to the semigroups.}
Applying the Euler--Yosida formula to
\eqref{eq:J-resolvent-intertwining} gives
\begin{equation}\label{eq:J-semigroup-intertwining}
  \bar{P}_tJ
  =
  \lim_{n\to\infty}
  \left(\frac nt\bar{R}_{n/t}\right)^nJ
  =
  J\lim_{n\to\infty}
  \left(\frac ntR_{n/t}\right)^n
  =
  JP_t.
\end{equation}
Equation \eqref{eq:I-resolvent-intertwining} shows that
\(Y=IC_0(\bar{L})\) is invariant under every resolvent
\(\bar{R}_\lambda\), hence under every Euler--Yosida approximation and
therefore under \(\bar{P}_t\).  The restricted semigroup
\[
  Q_t^{\bar{L}}:=I^{-1}\bar{P}_tI
\]
has resolvent \(R_\lambda^{\bar L}\), by
\eqref{eq:I-resolvent-intertwining}.  Its generator is therefore
\(\ms A_{\bar{L},\max}\), and
\begin{equation}\label{eq:I-semigroup-intertwining}
  \bar{P}_tI=IQ_t^{\bar{L}}.
\end{equation}
Consequently, for every \(t\geq0\), the corresponding semigroup
diagram also commutes:
\[
\begin{tikzcd}[column sep=large,row sep=large]
C(X)
  \arrow[r,"J",color=blue!70!black]
  \arrow[d,"P_t"',color=blue!70!black]
&
C(\bar{X}_L)
  \arrow[d,"\bar{P}_t"]
&
C_0(\bar{L})
  \arrow[l,"I"',color=red!70!black]
  \arrow[d,"Q_t^{\bar{L}}",color=red!70!black]
\\
C(X)
  \arrow[r,"J"',color=blue!70!black]
&
C(\bar{X}_L)
&
C_0(\bar{L})
  \arrow[l,"I",color=red!70!black]
\end{tikzcd}
\]
Evaluating the left square gives the pointwise comparison that will be
used in Step~6.  Namely, if \(x\in L\) and
\(\bar{x}=s^{-1}(x)\in\bar{L}\), then
\begin{equation}\label{eq:pointwise-J-semigroup}
  \bar{P}_t(Jf)(\bar{x})
  =
  (JP_tf)(\bar{x})
  =
  (P_tf)\bigl(\iota_L(x)\bigr).
\end{equation}
Thus \(P_t\) and \(\bar{P}_t\) are not being identified as operators
on the same function space.  Rather, \(J\) intertwines them, and
\eqref{eq:pointwise-J-semigroup} compares their values on the original
leaf \(L\) and its copy \(\bar{L}\).

\smallskip
\noindent\textbf{Step 5: identification with the intrinsic heat semigroup.}
Step~4 produced \(Q_t^{\bar{L}}\) abstractly by restricting the global
semigroup \(\bar{P}_t\) to \(Y\).  We prove that
\(Q_t^{\bar{L}}\) is the minimal heat semigroup of
\(\ms A_{\bar{L}}\) on the intrinsic manifold \(\bar{L}\).  This gives
the heat-kernel formula for initial data in \(C_0(\bar{L})\) and the
strict positivity of its kernel.  The extension to the bounded initial
data arising from \(C(X)\) is deferred to Step~6.

\smallskip
\noindent\emph{The intrinsic heat semigroup.}
Let \(k_{\bar{L}}\) be the smooth minimal heat kernel of
\(\partial_t-\ms A_{\bar{L}}\).  By the standard heat-kernel theory recalled
above, together with the uniform local geometry and coefficient bounds
inherited from the compact foliated space \(\bar X_L\), the operators
\[
  H_t^{\bar{L}}\varphi(\bar{x})
  =
  \int_{\bar{L}}
  k_{\bar{L}}(t,\bar{x},\bar{y})\varphi(\bar{y})\,
  dV_{g_{\ms A,\bar{L}}}(\bar{y})
\]
map \(C_0(\bar{L})\) into itself, are positive contractions, satisfy the
semigroup identity, and converge strongly to the identity as
\(t\downarrow0\).  Thus they form a strongly continuous semigroup on
\(C_0(\bar{L})\).  Their sub-Markov property gives
\[
  \int_{\bar{L}}k_{\bar{L}}(t,\bar{x},\bar{y})\,
  dV_{g_{\ms A,\bar{L}}}(\bar{y})\leq1.
\]

\smallskip
\noindent\emph{Identification through the resolvent.}
The resolvent of this semigroup,
\[
  S_\lambda\varphi
  =
  \int_0^\infty e^{-\lambda t}H_t^{\bar{L}}\varphi\,dt
\]
solves
\[
  (\lambda I-\ms A_{\bar{L},\max})S_\lambda\varphi=\varphi
\]
distributionally.  Indeed, for smooth compactly supported \(\varphi\),
integration of
\[
  -\frac{d}{dt}\bigl(e^{-\lambda t}H_t^{\bar{L}}\varphi\bigr)
  =
  e^{-\lambda t}
  (\lambda I-\ms A_{\bar{L}})H_t^{\bar{L}}\varphi
\]
from \(0\) to \(\infty\) gives the resolvent equation; approximation and
closedness give it for every \(\varphi\in C_0(\bar{L})\).

The solution is unique in \(C_0(\bar{L})\).  Indeed, suppose that
\(w\in\mc D(\ms A_{\bar{L},\max})\) satisfies
\[
  (\lambda I-\ms A_{\bar{L},\max})w=0.
\]
Leafwise elliptic regularity gives \(w\in C^2(\bar{L})\).  If
\(\sup_{\bar{L}}w>0\), then \(w\in C_0(\bar{L})\) attains a positive
maximum at some point \(\bar{x}\).  At this point,
\(\ms A_{\bar{L}}w(\bar{x})\leq0\), whereas the equation gives
\(\ms A_{\bar{L}}w(\bar{x})=\lambda w(\bar{x})>0\), a contradiction.
Applying the same argument to \(-w\) gives \(w=0\).  Consequently,
\[
  S_\lambda=R_\lambda^{\bar{L}}
  \qquad(\lambda>0).
\]
Since the resolvents agree, so do the semigroups.  The
semigroup--resolvent formula for \(Q_t^{\bar{L}}\) and the definition of
\(S_\lambda\) give, for every
\(\varphi\in C_0(\bar{L})\),
\[
  \int_0^\infty e^{-\lambda t}H_t^{\bar{L}}\varphi\,dt
  =
  S_\lambda\varphi
  =
  R_\lambda^{\bar{L}}\varphi
  =
  \int_0^\infty e^{-\lambda t}Q_t^{\bar{L}}\varphi\,dt
  \qquad(\lambda>0).
\]
Fix \(\bar{x}\in\bar{L}\).  The scalar function
\[
  t\longmapsto
  \bigl(H_t^{\bar{L}}\varphi-Q_t^{\bar{L}}\varphi\bigr)(\bar{x})
\]
is continuous and bounded, since both semigroups are strongly
continuous contractions.  The preceding identity says that its
Laplace transform vanishes for every \(\lambda>0\).  Uniqueness of the
Laplace transform therefore implies that this function vanishes for
all \(t\geq0\).  Since \(\varphi\) and \(\bar{x}\) were arbitrary,
\[
  H_t^{\bar{L}}=Q_t^{\bar{L}}
  \qquad(t\geq0).
\]
In particular,
\[
  Q_t^{\bar{L}}\varphi(\bar{x})
  =
  \int_{\bar{L}}
  k_{\bar{L}}(t,\bar{x},\bar{y})\varphi(\bar{y})\,
  dV_{g_{\ms A,\bar{L}}}(\bar{y})
  \quad
  \bigl(\varphi\in C_0(\bar{L})\bigr).
\]

\smallskip
\noindent\emph{Positivity and transport to \(L\).}
Since \(\bar{L}\) is connected, the parabolic strong maximum principle gives
\[
  k_{\bar{L}}(t,\bar{x},\bar{y})>0
  \quad(t>0,\ \bar{x},\bar{y}\in \bar{L}).
\]

Transport the kernel to \(L\) by
\begin{equation}\label{eq:transported-leaf-kernel}
  k_L(t,x,y)
  :=
  k_{\bar{L}}
  \bigl(t,s^{-1}(x),s^{-1}(y)\bigr).
\end{equation}
Because \(s\) is an isometry, this is the smooth minimal heat kernel of
\(\partial_t-\ms A_L\), no Jacobian occurs under the change of
variables, and the strict positivity just proved transfers to \(k_L\).

\smallskip
\noindent\textbf{Step 6: conservation of heat-kernel mass and bounded
global data.}
If \(f\in C(X)\), then
\[
  Jf|_{\bar L}=f\circ\iota_L\circ s
\]
is bounded and continuous on \(\bar L\), but it need not vanish at
infinity.  The \(C_0(\bar L)\)-formula from Step~5 therefore does not
yet apply.  We first prove, by a heat-kernel argument, that the minimal
heat kernel on \(\bar L\) has total mass one.

By \eqref{eq:A-metric-drift-decomposition}, completeness of the
principal metric, the uniform lower Ricci bound, and the boundedness of
the drift,
\[
  \ms A_{\bar L}=\Delta_{g_{\ms A,\bar L}}+W,
  \quad
  |W|_{g_{\ms A,\bar L}}\leq C.
\]
Fix
\(\bar o\in\bar L\) and set
\(r(\bar y)=d_{g_{\ms A,\bar L}}(\bar o,\bar y)\).  Away from the cut
locus, Laplacian comparison and the bound for \(W\) give
\[
  \ms A_{\bar L}r^2
  =
  2r\Delta_{g_{\ms A,\bar L}}r+2
  +2r\langle W,\nabla r\rangle
  \leq C(1+r^2).
\]
The same inequality holds in the barrier sense at the cut locus.
The standard smoothed-distance lemma therefore gives a smooth proper
function \(\widetilde r\colon\bar L\to[0,\infty)\) and a constant
\(C\) such that
\[
  |\widetilde r-r|\leq1,\quad
  |\nabla\widetilde r|\leq C,\quad
  \Delta_{g_{\ms A,\bar L}}\widetilde r
  \leq C(1+\widetilde r).
\]
Set \(\rho=1+\widetilde r^{\,2}\).  Then \(\rho\) is smooth and
proper, is comparable to \(1+r^2\), and the boundedness of \(W\) gives
a constant \(C_0\) such that
\begin{equation}\label{eq:heat-kernel-exhaustion}
  \ms A_{\bar L}\rho\leq C_0\rho.
\end{equation}

To apply \eqref{eq:heat-kernel-exhaustion} to the heat equation, define
\[
  m(t,\bar x)
  :=
  \int_{\bar L}
  k_{\bar L}(t,\bar x,\bar y)\,
  dV_{g_{\ms A,\bar L}}(\bar y)
  \quad(t>0).
\]
The minimal heat kernel is obtained as the increasing limit of the
Dirichlet heat kernels on a smooth exhaustion of \(\bar L\).  The
parabolic maximum principle on each exhaustion domain therefore gives
\[
  0\leq m(t,\bar x)\leq1.
\]
Moreover,
\begin{equation}\label{eq:mass-heat-equation}
  (\partial_t-\ms A_{\bar L})m=0
  \quad\text{on }(0,\infty)\times\bar L,
  \quad
  m(t,\cdot)\longrightarrow1
  \quad\text{locally uniformly as }t\downarrow0.
\end{equation}
For clarity, the initial limit follows without assuming conservation.
If \(K\Subset\bar L\), choose a relatively compact smooth domain
\(\Omega\) containing \(K\).  If \(k_\Omega\) is its Dirichlet heat
kernel, minimality gives
\[
  m(t,\bar x)
  \geq
  \int_\Omega k_\Omega(t,\bar x,\bar y)\,
  dV_{g_{\ms A,\bar L}}(\bar y).
\]
The expression on the right tends uniformly to \(1\) on \(K\) as
\(t\downarrow0\), whereas \(m\leq1\).  This proves the second assertion
in \eqref{eq:mass-heat-equation}.  The first follows by taking the
increasing limit of the Dirichlet solutions and applying local
parabolic regularity.

Set
\[
  u(t,\bar x):=1-m(t,\bar x).
\]
Extend \(u\) to \(\{0\}\times\bar L\) by setting
\(u(0,\bar x)=0\).  Then \(0\leq u\leq1\),
\[
  (\partial_t-\ms A_{\bar L})u=0,
  \quad
  u(t,\cdot)\longrightarrow0
  \quad\text{locally uniformly as }t\downarrow0.
\]
Choose \(C_1>C_0\) and put
\[
  \Phi(t,\bar x):=e^{C_1t}\rho(\bar x).
\]
By \eqref{eq:heat-kernel-exhaustion},
\[
  (\partial_t-\ms A_{\bar L})\Phi
  =
  e^{C_1t}\bigl(C_1\rho-\ms A_{\bar L}\rho\bigr)
  \geq
  (C_1-C_0)e^{C_1t}\rho>0.
\]
Fix \(T>0\) and \(\varepsilon>0\).  Since \(\rho\) is proper and
\(u\leq1\),
\[
  u(t,\bar x)-\varepsilon\Phi(t,\bar x)
  \longrightarrow-\infty
\]
as \(\bar x\) tends to infinity, uniformly for \(0\leq t\leq T\).
At \(t=0\) this function equals
\(-\varepsilon\rho<0\).  If it had a positive maximum on
\([0,T]\times\bar L\), properness would ensure that the maximum was
attained at some \((t_0,\bar x_0)\) with \(t_0>0\).  At this point the
parabolic maximum principle would give
\[
  (\partial_t-\ms A_{\bar L})
  (u-\varepsilon\Phi)(t_0,\bar x_0)\geq0.
\]
On the other hand,
\[
  (\partial_t-\ms A_{\bar L})(u-\varepsilon\Phi)
  =
  -\varepsilon
  (\partial_t-\ms A_{\bar L})\Phi<0,
\]
a contradiction.  Hence \(u\leq\varepsilon\Phi\) on
\([0,T]\times\bar L\).  Letting \(\varepsilon\downarrow0\) gives
\(u=0\).  Since \(T\) was arbitrary, \(m\equiv1\); that is,
\begin{equation}\label{eq:copied-kernel-conservative}
  \int_{\bar L}
  k_{\bar L}(t,\bar x,\bar y)\,
  dV_{g_{\ms A,\bar L}}(\bar y)=1
  \quad(t>0).
\end{equation}

It remains to pass from \(C_0(\bar L)\) to the bounded functions copied
from \(X\).  Since \(\bar P_t\) is positive and fixes \(1\), the Riesz
representation theorem gives a probability measure
\(\bar\pi_t(\bar x,\cdot)\) on \(\bar X_L\) such that
\[
  \bar P_tF(\bar x)
  =
  \int_{\bar X_L}F\,d\bar\pi_t(\bar x,\cdot)
  \quad\bigl(F\in C(\bar X_L)\bigr).
\]
For every \(\varphi\in C_0(\bar L)\), Steps~4--5 give
\begin{align*}
  \int_{\bar X_L}I\varphi\,d\bar\pi_t(\bar x,\cdot)
  &=
  \bar P_t(I\varphi)(\bar x)\\
  &=
  Q_t^{\bar L}\varphi(\bar x)\\
  &=
  \int_{\bar L}
  k_{\bar L}(t,\bar x,\bar y)\varphi(\bar y)\,
  dV_{g_{\ms A,\bar L}}(\bar y).
\end{align*}
Because \(\bar L\) is open in \(\bar X_L\) and \(I\) is zero extension,
uniqueness in the Riesz representation theorem on the locally compact
space \(\bar L\) yields
\[
  \bar\pi_t(\bar x,\cdot)|_{\bar L}
  =
  k_{\bar L}(t,\bar x,\bar y)\,
  dV_{g_{\ms A,\bar L}}(\bar y).
\]
The right-hand side has total mass one by
\eqref{eq:copied-kernel-conservative}; hence
\(\bar\pi_t(\bar x,X)=0\).  It follows that, for every \(f\in C(X)\),
\begin{equation}\label{eq:bounded-copied-leaf-formula}
  \bar{P}_t(Jf)(\bar{x})
  =
  \int_{\bar{L}}
  k_{\bar{L}}(t,\bar{x},\bar{y})Jf(\bar{y})\,
  dV_{g_{\ms A,\bar{L}}}(\bar{y}).
\end{equation}
This proves the required extension from \(C_0(\bar L)\) to the
canonical bounded initial values coming from \(C(X)\).

Set \(\bar{x}=s^{-1}(x)\).  Each equality in the following computation
is justified by the indicated identity:
\begin{align*}
  (P_tf)\bigl(\iota_L(x)\bigr)
  &=
  \bar{P}_t(Jf)(\bar{x})
  &&\text{by \eqref{eq:pointwise-J-semigroup}},\\
  &=
  \int_{\bar{L}}
  k_{\bar{L}}(t,\bar{x},\bar{y})
  f\bigl(\iota_L(s(\bar{y}))\bigr)\,
  dV_{g_{\ms A,\bar{L}}}(\bar{y})
  &&\text{by \eqref{eq:bounded-copied-leaf-formula} and
  \eqref{eq:duplicated-leaf-embeddings}},\\
  &=
  \int_L k_L(t,x,y)f\bigl(\iota_L(y)\bigr)\,
  dV_{g_{\ms A}}(y)
  &&\text{by \(y=s(\bar{y})\),
  \eqref{eq:copied-volume-form}, and
  \eqref{eq:transported-leaf-kernel}}.
\end{align*}
Identifying \(L\) with its image under \(\iota_L\), this proves
\eqref{eq:leafwise-kernel-formula}.  Transporting
\eqref{eq:copied-kernel-conservative} through \(s\), and combining it
with the strict positivity proved in Step~5, gives
\eqref{eq:kernel-positive-conservative}.

Finally, let \(\pi_t(\iota_L(x),\cdot)\) be the transition probability
of \(P_t\) from \(\iota_L(x)\).  The formula just proved says that, for
every \(f\in C(X)\),
\[
  \int_X f\,d\pi_t(\iota_L(x),\cdot)
  =
  \int_X f\,d(\iota_L)_*
  \bigl(k_L(t,x,y)\,dV_{g_{\ms A}}(y)\bigr).
\]
Uniqueness in the Riesz representation theorem gives the last assertion
of part~\textup{(iii)}.
\end{proof}

\begin{remark}\label{rem:compactness-versus-leaves}
There are two distinct conservation statements in the proof.  The
identity \(P_t1=1\) concerns the global semigroup on the compact
foliated space \(X\); it follows from the resolvent construction in
part~\textup{(ii)}.  It does not by itself show that the minimal heat
kernel on a noncompact leaf has total mass one.  Indeed, Step~5
identifies the two semigroups only on \(C_0(\bar L)\), and the constant
function \(1\) does not belong to \(C_0(\bar L)\) when \(\bar L\) is
noncompact.

Step~6 proves the second conservation statement directly on the leaf.
Compactness of \(X\) gives uniform coefficient bounds; these imply
completeness of the principal leafwise metric, a uniform lower Ricci
bound, and a uniformly bounded drift.  Laplacian comparison then gives
a proper exhaustion \(\rho\) satisfying
\(\ms A_{\bar L}\rho\leq C\rho\), and the parabolic maximum principle
shows that
\[
  \int_{\bar L}k_{\bar L}(t,\bar x,\bar y)\,
  dV_{g_{\ms A,\bar L}}(\bar y)=1.
\]
Thus the leafwise conclusion is obtained from the geometry already
available in the foliated setting; no compactness or proper embedding
of the individual leaf is required.  Once this is known, the
\(C_0(\bar L)\)-identity extends to the bounded functions copied from
\(X\), proving that the global transition probability remains in the
initial intrinsic leaf.
\end{remark}

\subsection{The classical domain as a core}
\label{subsec:candel-core}

Let \(G\) be a closed operator on a Banach space.  A linear subspace
\(\mc C\subset\mc D(G)\) is a \emph{core} for \(G\) if it is dense in
\(\mc D(G)\) for the graph norm
\[
  \|u\|_G:=\|u\|+\|Gu\|.
\]
Equivalently, the closure of \(G|_{\mc C}\) is \(G\).  We need this
notion because infinitesimal stationarity is initially known only on
the classical leafwise domain, whereas semigroup invariance is governed
by the full generator.

\begin{lemma}\label{lem:classical-core}
The classical domain \(\mc D_0(\ms A)\) is a core for
\(\ms A_{\max}\).
\end{lemma}

\begin{proof}
Fix \(\lambda>0\) and \(u\in\mc D(\ms A_{\max})\), and put
\[
  q=(\lambda I-\ms A_{\max})u.
\]
Choose \(q_j\in C^{2,0}(X)\) with \(q_j\to q\) uniformly and define
\[
  u_j=R_\lambda q_j.
\]
Theorem~\ref{thm:candel-diffusion}\textup{(i)} gives
\(u_j\in\mc D_0(\ms A)\).  Since \(u=R_\lambda q\), the resolvent
estimate gives
\[
  \|u_j-u\|_\infty
  \leq
  \frac1\lambda\|q_j-q\|_\infty
  \longrightarrow0.
\]
Moreover,
\[
  \ms A_{\max}u_j
  =
  \ms Au_j
  =
  \lambda u_j-q_j,
\]
and therefore
\[
  \|\ms A_{\max}u_j-\ms A_{\max}u\|_\infty
  \leq
  \lambda\|u_j-u\|_\infty+\|q_j-q\|_\infty
  \longrightarrow0.
\]
Thus \(u_j\to u\) in the graph norm of \(\ms A_{\max}\).
\end{proof}

Uniform density of \(C^{2,0}(X)\) alone controls only the functions.
Resolvent approximation simultaneously controls their images under the
generator.

\subsection{Stationarity and saturation of support}
\label{subsec:candel-stationarity}

The core property relates the infinitesimal and semigroup formulations
of stationarity.  Strict positivity of the leafwise heat kernel then
determines the support of a stationary measure.

A Borel probability measure \(\nu\) on \(X\) is
\emph{stationary} for \(\{P_t\}_{t\geq0}\) if
\[
  \int_XP_tf\,d\nu=\int_Xf\,d\nu
  \quad
  \bigl(f\in C(X),\ t\geq0\bigr).
\]

\begin{proposition}
\label{prop:stationarity-equivalence}
For a Borel probability measure \(\nu\) on \(X\), the following
conditions are equivalent:
\begin{enumerate}[label=\textup{(\roman*)}]
\item
\[
  \int_X\ms Af\,d\nu=0
  \quad\text{for every }f\in\mc D_0(\ms A);
\]
\item \(\nu\) is stationary for \(\{P_t\}_{t\geq0}\).
\end{enumerate}
\end{proposition}

\begin{proof}
Assume \textup{(i)}.  If \(u\in\mc D(\ms A_{\max})\), choose
\(u_j\in\mc D_0(\ms A)\) such that
\[
  u_j\longrightarrow u,
  \quad
  \ms Au_j\longrightarrow\ms A_{\max}u
\]
uniformly.  Since \(\nu\) is finite,
\[
  \int_X\ms A_{\max}u\,d\nu
  =
  \lim_{j\to\infty}\int_X\ms Au_j\,d\nu
  =
  0.
\]

Because \(\ms A_{\max}\) is the generator of \(\{P_t\}\), every
\(u\in\mc D(\ms A_{\max})\) satisfies
\[
  P_tu\in\mc D(\ms A_{\max}),
  \quad
  \frac{d}{dt}P_tu=\ms A_{\max}P_tu
  \quad(t\geq0).
\]
Consequently,
\[
  \frac{d}{dt}\int_XP_tu\,d\nu
  =
  \int_X\ms A_{\max}P_tu\,d\nu=0.
\]
Hence
\[
  \int_XP_tu\,d\nu=\int_Xu\,d\nu
  \quad
  \bigl(u\in\mc D(\ms A_{\max})\bigr).
\]
The generator domain is dense in \(C(X)\).  Given \(f\in C(X)\), choose
\(u_j\in\mc D(\ms A_{\max})\) with \(u_j\to f\) uniformly.  Since
\(\|P_t\|\leq1\),
\[
  \left|
    \int_XP_tf\,d\nu-\int_Xf\,d\nu
  \right|
  \leq
  2\|f-u_j\|_\infty
  \longrightarrow0.
\]
Thus \(\nu\) is stationary.

Conversely, suppose \(\nu\) is stationary.  If
\(f\in\mc D_0(\ms A)\), then
\[
  \frac{P_tf-f}{t}\longrightarrow\ms Af
  \quad\text{uniformly as }t\downarrow0.
\]
Integrating and using stationarity proves \textup{(i)}.
\end{proof}

This is Candel's classical infinitesimal stationarity criterion (see
\cite{Candel2003}, Proposition~5.1).  The core argument above allows the
test functions to range over the full classical domain
\(\mc D_0(\ms A)\).

A subset \(E\subset X\) is \emph{leaf-saturated} if it contains the
entire leaf through each of its points.

\begin{proposition}
\label{prop:stationary-support-saturated}
The support of every stationary probability measure is leaf-saturated.
More precisely, if \(\nu\) is stationary and \(x\in\supp\nu\), then the
leaf \(L(x)\) through \(x\) is contained in \(\supp\nu\).
\end{proposition}

\begin{proof}
Let \(y\in L(x)\), and let \(U\subset X\) be an open neighborhood of
\(y\).  Choose open sets \(V,W\subset X\) such that
\[
  y\in V\subset\overline V\subset W
  \subset\overline W\subset U.
\]
Since \(X\) is compact and metrizable, Urysohn's lemma gives
\(\varphi\in C(X)\) such that
\[
  0\leq\varphi\leq1,\quad
  \varphi=1\quad\text{on }\overline V,\quad
  \varphi=0\quad\text{on }X\setminus W.
\]
In particular, \(\supp\varphi\subset\overline W\subset U\).
The restriction of \(\varphi\) to the intrinsic leaf \(L(x)\) is
positive on a nonempty leafwise open set.  Strict positivity of the
heat kernel in
Theorem~\ref{thm:candel-diffusion}\textup{(iii)} therefore gives
\[
  P_t\varphi(x)>0
  \quad(t>0).
\]
Because \(P_t\varphi\) is continuous, there are a neighborhood
\(O\) of \(x\) and a constant \(c>0\) such that
\(P_t\varphi\geq c\) on \(O\).  Since \(x\in\supp\nu\),
\(\nu(O)>0\).  Stationarity now gives
\[
  \nu(U)
  \geq\int_X\varphi\,d\nu
  =\int_XP_t\varphi\,d\nu
  \geq c\,\nu(O)>0.
\]
Every neighborhood of \(y\) has positive measure, so
\(y\in\supp\nu\).  Since \(y\) was arbitrary,
\(L(x)\subset\supp\nu\).
\end{proof}

\begin{remark}
The proof uses the intrinsic leaf topology to propagate positivity and
the ambient topology of \(X\) to use the Feller property and define
\(\supp\nu\).  No properness of the leaf inclusion is needed.  In
particular, if \(L(x)\) is dense in \(X\) and
\(x\in\supp\nu\), then the closedness of the support forces
\(\supp\nu=X\).
\end{remark}

\appendix

\section{The \texorpdfstring{\(p\)}{p}-Laplacian rigidity theorem}
\label{app:p-rigidity}

This appendix proves Theorem~\ref{thm:p-rigidity}.  We treated the case
\(p=2\) in the main text because the quadratic defect displays the
geometry of the argument most clearly.  For general \(p\), its role is
played by the Bregman divergence associated with \(y\mapsto |y|^p\).

In the definition of \(\lambda_{1,p}(X)\), it suffices to use
nonnegative test functions.  Indeed, for \(v\in C_c^\infty(X)\), the
functions
\[
  F_\delta(v)=\sqrt{v^2+\delta^2}-\delta
\]
are smooth, nonnegative, and supported in \(\supp v\).  Moreover, they
converge to \(|v|\) in \(L^p\) and in \(p\)-energy as
\(\delta\downarrow0\), by dominated convergence and the bounds
\[
  |F_\delta(v)|\leq |v|,
  \qquad |\nabla F_\delta(v)|\leq |\nabla v|.
\]

An exact \(p\)-defect identity gives the lower bound, while equality
forces the defect to vanish.  For \(\rho=u^p\), this yields the same
linear transport equation as in the quadratic case, with coefficient
\(m-1\).  The leafwise-diffusion and support-saturation arguments
therefore apply without invoking a nonlinear \(p\)-heat flow.

\subsection{The nonlinear defect and critical sequences}

Fix \(p\in(1,\infty)\).  For two vectors \(x,y\) in the same
finite-dimensional inner-product space, define the convexity defect
\begin{equation}\label{eq:p-convexity-defect}
  \mc G_p(x,y)
  :=|x|^p+(p-1)|y|^p
    -p\langle |y|^{p-2}y,x\rangle,
\end{equation}
where \(|0|^{p-2}0\) is understood to be zero.
For \(p=2\), this reduces to
\(\mc G_2(x,y)=|x-y|^2\), the quadratic defect used in the main
argument.

\begin{lemma}\label{lem:p-uniform-convexity}
There is a constant \(c_p>0\), depending only on \(p\), such that
\(\mc G_p(x,y)\geq0\) for all \(x,y\), with equality if and only if
\(x=y\).  Moreover,
\[
  \mc G_p(x,y)\geq
  \begin{cases}
    c_p|x-y|^p,
      &p\geq2,\\[2mm]
    c_p(|x|+|y|)^{p-2}|x-y|^2,
      &1<p<2.
  \end{cases}
\]
The expression in the second line is defined to be zero when \(x=y=0\).
One may take
\[
  c_p=
  \begin{cases}
    p/4^p, & p\geq2,\\[1mm]
    p(p-1)/2, & 1<p<2.
  \end{cases}
\]
\end{lemma}

\begin{proof}
Let \(\Phi(z)=|z|^p\).  Then
\[
  \mc G_p(x,y)
  =\Phi(x)-\Phi(y)-D\Phi(y)[x-y],
\]
so the strict convexity of \(\Phi\) gives
\(\mc G_p(x,y)\geq0\), with equality precisely when \(x=y\).

Set \(e=x-y\).  There is nothing more to prove when \(e=0\), so assume
that \(e\neq0\).  Taylor's formula with integral remainder gives
\begin{equation}\label{eq:p-convexity-remainder}
  \mc G_p(x,y)
  =\int_0^1(1-t)
    D^2\Phi(y+te)[e,e]\,dt.
\end{equation}
For \(z\neq0\), decompose a vector \(v\) as
\[
  v=v_\perp+v_\parallel,
  \qquad
  v_\parallel=\frac{\langle v,z\rangle}{|z|^2}z,
  \qquad
  v_\perp\perp z.
\]
A direct computation gives
\[
  D^2\Phi(z)[v,v]
  =p|z|^{p-2}
    \bigl(|v_\perp|^2+(p-1)|v_\parallel|^2\bigr).
\]
If \(1<p<2\) and \(y+te=0\) at \(t=t_0\), the integrand in
\eqref{eq:p-convexity-remainder} grows at most like
\(|t-t_0|^{p-2}\).  Since \(p>1\), this singularity is integrable, and
we interpret \eqref{eq:p-convexity-remainder} as an improper integral.

\begin{enumerate}[label=\textup{(\roman*)}]
\item Suppose that \(p\geq2\).  Let
\(I=\{t\in[0,1]:|y+te|<|e|/4\}\).  This is an interval, and for
\(s,t\in I\),
\[
  |t-s||e|
  \leq |y+te|+|y+se|
  <\frac{|e|}{2},
\]
so \(I\) has length at most \(1/2\).  Hence
\(J=[0,3/4]\setminus I\) has measure at least \(1/4\), and every
\(t\in J\) satisfies \(1-t\geq1/4\) and
\(|y+te|\geq|e|/4\).
Since \(p-1\geq1\), the Hessian formula gives
\[
  D^2\Phi(z)[v,v]\geq p|z|^{p-2}|v|^2.
\]
Therefore \eqref{eq:p-convexity-remainder} yields
\[
  \mc G_p(x,y)
  \geq \int_J(1-t)p|y+te|^{p-2}|e|^2\,dt
  \geq \frac{p}{4^p}|e|^p.
\]

\item Suppose that \(1<p<2\).  Since \(p-1<1\), the Hessian formula
gives
\[
  D^2\Phi(z)[v,v]\geq p(p-1)|z|^{p-2}|v|^2.
\]
Moreover,
\[
  |y+te|
  =|(1-t)y+tx|
  \leq(1-t)|y|+t|x|
  \leq|x|+|y|.
\]
Since \(p-2<0\), it follows that
\[
  D^2\Phi(y+te)[e,e]
  \geq p(p-1)(|x|+|y|)^{p-2}|e|^2
\]
whenever \(y+te\neq0\).  Using the improper-integral interpretation at
a possible crossing, \eqref{eq:p-convexity-remainder} gives
\[
  \mc G_p(x,y)
  \geq\frac{p(p-1)}{2}(|x|+|y|)^{p-2}|e|^2,
\]
which proves the second estimate.
\end{enumerate}
\end{proof}

At the critical McKean scale, set
\begin{equation}\label{eq:p-critical-parameter}
  \alpha=\frac{m-1}{p}.
\end{equation}
Then \(p\alpha=m-1\), so the transport equation for the density
\(\rho=u^p\) has coefficient \(m-1\), independent of \(p\).

\begin{lemma}\label{lem:p-mckean-defect}
Let \((X^m,g)\) be a Riemannian manifold without boundary.  If
\(B\in C^2(X)\) satisfies \(|\nabla B|=1\), then every nonnegative
\(u\in C_c^\infty(X)\) satisfies the exact identity
\begin{equation}\label{eq:p-mckean-defect}
\begin{aligned}
  \int_X\bigl(|\nabla u|^p-\alpha^pu^p\bigr)\,dV
  ={}&\int_X\mc G_p(\nabla u,-\alpha u\nabla B)\,dV\\
     &+\alpha^{p-1}\int_X
       \bigl(\Delta B-(m-1)\bigr)u^p\,dV.
\end{aligned}
\end{equation}
\end{lemma}

\begin{proof}
Because \(u\geq0\), the function \(u^p\) belongs to \(C_c^1(X)\) and
\(\nabla(u^p)=pu^{p-1}\nabla u\).  Using \(|\nabla B|=1\), we obtain
the pointwise identity
\[
  \mc G_p(\nabla u,-\alpha u\nabla B)
  =|\nabla u|^p+(p-1)\alpha^pu^p
   +\alpha^{p-1}\langle\nabla B,\nabla(u^p)\rangle.
\]
Since \(u\) is compactly supported, integration by parts gives
\[
  \int_X\langle\nabla B,\nabla(u^p)\rangle\,dV
  =-\int_X(\Delta B)u^p\,dV.
\]
Consequently,
\[
\begin{aligned}
 &\int_X\mc G_p(\nabla u,-\alpha u\nabla B)\,dV
  +\alpha^{p-1}\int_X
    \bigl(\Delta B-(m-1)\bigr)u^p\,dV\\
 &\qquad
  =\int_X|\nabla u|^p\,dV
   +\bigl((p-1)\alpha^p-\alpha^{p-1}(m-1)\bigr)
      \int_Xu^p\,dV\\
 &\qquad
  =\int_X\bigl(|\nabla u|^p-\alpha^pu^p\bigr)\,dV,
\end{aligned}
\]
where the last equality uses \(m-1=p\alpha\).
\end{proof}

Assume now that \(\Delta B\geq m-1\).  Then
\eqref{eq:p-mckean-defect} expresses the spectral excess as the sum of
two nonnegative defects.  At the critical value
\(\lambda_{1,p}(X)=\alpha^p\), both defects must vanish along every
normalized minimizing sequence.  Lemma~\ref{lem:p-uniform-convexity}
turns the first vanishing defect into strong \(L^p\) control, and the
change of variables \(\rho=u^p\) converts this control into the linear
\(L^1\) transport equation used on the suspension.

\begin{proposition}\label{prop:p-critical-sequence}
Let \((X^m,g)\), \(m\geq2\), be complete, and let \(B\in C^2(X)\)
satisfy
\[
  |\nabla B|=1,
  \quad
  \Delta B\geq m-1,
\]
and suppose that
\[
  \lambda_{1,p}(X)=\alpha^p.
\]
Then there exists a nonnegative normalized minimizing sequence
\(\{u_j\}\subset C_c^\infty(X)\) satisfying
\[
  \int_Xu_j^p\,dV=1,
  \quad
  \int_X|\nabla u_j|^p\,dV\longrightarrow\alpha^p.
\]
Every such sequence has the following properties.
\begin{enumerate}[label=\textup{(\roman*)}]
\item Both defects vanish:
\begin{equation}\label{eq:p-two-defects}
\begin{split}
  \int_X\mc G_p(\nabla u_j,-\alpha u_j\nabla B)\,dV
  &\longrightarrow0,\\
  \int_X\bigl(\Delta B-(m-1)\bigr)u_j^p\,dV
  &\longrightarrow0.
\end{split}
\end{equation}
\item The first-order defect vanishes strongly in \(L^p(X)\):
\begin{equation}\label{eq:p-first-order-convergence}
  \|\nabla u_j+\alpha u_j\nabla B\|_{L^p(X)}
  \longrightarrow0.
\end{equation}
\item The probability densities \(\rho_j=u_j^p\) satisfy the
asymptotic transport equation
\begin{equation}\label{eq:p-density-transport}
  \|\nabla\rho_j+(m-1)\rho_j\nabla B\|_{L^1(X)}
  \longrightarrow0.
\end{equation}
\end{enumerate}
\end{proposition}

\begin{proof}
Existence follows from the reduction to nonnegative test functions at
the beginning of the appendix.  Fix any such sequence.  Applying
Lemma~\ref{lem:p-mckean-defect} and using the normalization, we obtain
\[
\begin{aligned}
  \int_X|\nabla u_j|^p\,dV-\alpha^p
  ={}&\int_X\mc G_p(\nabla u_j,-\alpha u_j\nabla B)\,dV\\
     &+\alpha^{p-1}\int_X
       \bigl(\Delta B-(m-1)\bigr)u_j^p\,dV.
\end{aligned}
\]
The left-hand side tends to zero, while both terms on the right are
nonnegative.  Since \(m\geq2\), we have \(\alpha>0\); hence both defect
integrals tend to zero.  This proves \textup{(i)}.

To prove \textup{(ii)}, set
\[
  R_j=\nabla u_j+\alpha u_j\nabla B,
  \quad
  A_j=|\nabla u_j|+\alpha u_j.
\]
\begin{enumerate}[label=\textup{(\arabic*)}]
\item If \(p\geq2\), the first estimate in
Lemma~\ref{lem:p-uniform-convexity} gives
\[
  c_p\|R_j\|_{L^p(X)}^p
  \leq
  \int_X\mc G_p(\nabla u_j,-\alpha u_j\nabla B)\,dV
  \longrightarrow0.
\]

\item If \(1<p<2\), the second estimate in
Lemma~\ref{lem:p-uniform-convexity} and \textup{(i)} give
\[
  I_j:=\int_XA_j^{p-2}|R_j|^2\,dV,
  \qquad
  c_pI_j
  \leq\int_X\mc G_p(\nabla u_j,-\alpha u_j\nabla B)\,dV
  \longrightarrow0.
\]
The integrand is interpreted as zero where \(A_j=0\), since then
\(R_j=0\).  Moreover,
\[
  \int_XA_j^p\,dV
  \leq 2^{p-1}
    \left(\int_X|\nabla u_j|^p\,dV+\alpha^p\right),
\]
so the sequence \(\{A_j\}\) is bounded in \(L^p(X)\).  H\"older's
inequality, with exponents \(2/p\) and \(2/(2-p)\), now gives
\[
\begin{aligned}
  \int_X|R_j|^p\,dV
  &=\int_X
    \bigl(A_j^{p-2}|R_j|^2\bigr)^{p/2}
    A_j^{p(2-p)/2}\,dV\\
  &\leq I_j^{p/2}
    \left(\int_XA_j^p\,dV\right)^{(2-p)/2}
  \longrightarrow0.
\end{aligned}
\]
\end{enumerate}
Thus \(R_j\to0\) in \(L^p(X)\) in both cases, proving \textup{(ii)}.

To prove \textup{(iii)}, set \(\rho_j=u_j^p\).  Since \(p>1\), the map
\(s\mapsto s^p\) is \(C^1\) on \([0,\infty)\), so
\(\rho_j\in C_c^1(X)\).  The chain rule and \(p\alpha=m-1\) give
\[
  \nabla\rho_j+(m-1)\rho_j\nabla B
  =pu_j^{p-1}R_j.
\]
Hence, by H\"older's inequality and \(\|u_j\|_{L^p}=1\),
\[
  \|\nabla\rho_j+(m-1)\rho_j\nabla B\|_{L^1(X)}
  \leq p\|u_j\|_{L^p(X)}^{p-1}\|R_j\|_{L^p(X)}
  =p\|R_j\|_{L^p(X)}
  \longrightarrow0.
\]
The last convergence follows from \textup{(ii)} and proves
\textup{(iii)}.
\end{proof}

\subsection{The sharp bound and the equality case}

\begin{proof}[Proof of Theorem~\ref{thm:p-rigidity}]
\emph{The lower bound.}
Fix \(p\in(1,\infty)\) and recall that \(\alpha=(m-1)/p\).
Since \(M\) is compact, its sectional curvature is also bounded below.
Thus there is \(a\geq1\) such that
\[
  -a^2\leq\sec_{\wti g}\leq-1.
\]
Proposition~\ref{prop:busemann-comparison} therefore gives, for every
\(\xi\in\partial_\infty\wti M\),
\[
  |\nabla B_\xi|=1,
  \quad
  \Delta B_\xi\geq m-1.
\]
Fix one such \(\xi\).  For every nonnegative
\(u\in C_c^\infty(\wti M)\), Lemma~\ref{lem:p-mckean-defect} and the
nonnegativity of both defects give
\[
  \int_{\wti M}|\nabla u|^p\,dV
  \geq\alpha^p\int_{\wti M}u^p\,dV.
\]
Taking the infimum and using the reduction to nonnegative test
functions yields
\[
  \lambda_{1,p}(\wti M)\geq\alpha^p.
\]

\emph{Rigidity in the equality case.}
Suppose that \(\lambda_{1,p}(\wti M)=\alpha^p\).
Fix \(\xi_0\in\partial_\infty\wti M\) and write
\(B_0=B_{\xi_0}\).  By
Proposition~\ref{prop:p-critical-sequence}, there is a nonnegative
normalized minimizing sequence \(u_j\).  Properties \textup{(i)} and
\textup{(iii)} of that proposition, with \(\rho_j=u_j^p\), give
\[
  \int_{\wti M}
    \bigl(\Delta B_0-(m-1)\bigr)\rho_j\,dV\longrightarrow0,
  \quad
  \|\nabla\rho_j+(m-1)\rho_j\nabla B_0\|_{L^1(\wti M)}
  \longrightarrow0.
\]
Let
\[
  Z=(\wti M\times\partial_\infty\wti M)/\Gamma,
  \quad
  \iota_{\xi_0}(x)=[x,\xi_0],
\]
and define probability measures on \(Z\) by
\[
  \mu_j=(\iota_{\xi_0})_*(\rho_jdV).
\]
Since \(Z\) is compact, a subsequence converges weakly to a probability
measure \(\mu\).  Because the Busemann defect \(D\) is continuous,
\[
  \int_ZD\,d\mu
  =\lim_{j\to\infty}\int_ZD\,d\mu_j
  =\lim_{j\to\infty}\int_{\wti M}
    \bigl(\Delta B_0-(m-1)\bigr)\rho_j\,dV
  =0.
\]
Since \(D\geq0\), it follows that
\begin{equation}\label{eq:p-limit-support}
  \supp\mu\subseteq D^{-1}(0).
\end{equation}

We next show that \(\supp\mu\) is leaf-saturated.  The transport
coefficient is \(m-1\), independent of \(p\), so the relevant operator
is the same linear leafwise operator used in
Section~\ref{sec:stationarity}:
\[
  L=\Delta^{\mathrm{leaf}}
    -(m-1)\langle V,\nabla^{\mathrm{leaf}}\,\cdot\,\rangle.
\]
Thus no nonlinear \(p\)-heat operator is needed.  For
\(f\in\mc D_0(A)\),
set \(f_0=f\circ\iota_{\xi_0}\).  Since
\(\rho_j\in C_c^1(\wti M)\) and \(f_0\in C^2(\wti M)\), integration by
parts for the compactly supported vector field
\(\rho_j\nabla f_0\) gives
\[
  \int_ZLf\,d\mu_j
  =-\int_{\wti M}
    \left\langle
      \nabla\rho_j+(m-1)\rho_j\nabla B_0,
      \nabla f_0
    \right\rangle dV.
\]
By Lemma~\ref{lem:classical-gradient-control},
\(\|\nabla f_0\|_\infty<\infty\).  Therefore
\[
  \left|\int_ZLf\,d\mu_j\right|
  \leq
  \|\nabla f_0\|_\infty
  \|\nabla\rho_j+(m-1)\rho_j\nabla B_0\|_{L^1(\wti M)}
  \longrightarrow0.
\]
Since \(Lf\) is continuous on \(Z\), weak convergence also gives
\[
  \int_ZLf\,d\mu=0
  \quad\text{for every }f\in\mc D_0(A).
\]
Lemma~\ref{lem:saturated-support}\textup{(ii)}--\textup{(iii)} now gives
stationarity of \(\mu\) and leaf saturation of its support.

Since \(\mu\) is a probability measure, its support is nonempty.  Choose
\(z_*=[x_*,\xi_*]\in\supp\mu\).  Leaf saturation implies that
\([x,\xi_*]\in\supp\mu\) for every \(x\in\wti M\).  Together with
\eqref{eq:p-limit-support}, this gives
\[
  [x,\xi_*]\in\supp\mu\subseteq D^{-1}(0).
\]
By the definition of \(D\), it follows that
\[
  \Delta B_{\xi_*}=m-1
\]
throughout \(\wti M\).  Moreover,
\(|\nabla B_{\xi_*}|=1\) by
Proposition~\ref{prop:busemann-comparison}.
Since \(M\) is closed, \(\sec_{\wti g}\) also has a uniform lower bound.
Lemma~\ref{lem:riemannian-rigidity} therefore gives
\[
  (\wti M,\wti g)\cong\bH^m(-1).
\]

\emph{The hyperbolic converse.}
Suppose that \(\wti M\cong\bH^m(-1)\).  The lower bound has already
been proved, so it remains only to establish the reverse inequality.
We use the radial-cutoff argument from the quadratic case, with
\(\alpha=(m-1)/p\).  Fix \(o\in\bH^m(-1)\), set \(r=d(o,\cdot)\), and
choose \(\eta\in C^\infty(\mathbb R)\) such that
\[
  0\leq\eta\leq1,\quad
  \eta(s)=0\ \text{for }s\leq0,\quad
  \eta(s)=1\ \text{for }s\geq1.
\]
For \(R>2\), define
\[
  \chi_R(r):=\eta(r-1)\eta(R+1-r),
  \quad
  u_R:=\chi_R(r)\ee^{-\alpha r}.
\]
Because \(\chi_R\) vanishes for \(r\leq1\), the nonsmoothness of the
distance function at \(o\) causes no problem, and
\(u_R\in C_c^\infty(\bH^m(-1))\).  Moreover,
\(\chi_R=1\) on \(2\leq r\leq R\), while its two transition regions
\([1,2]\) and \([R,R+1]\) have fixed width and
\(\|\chi_R'\|_\infty\) is independent of \(R\).

In polar coordinates,
\[
  dV=\sinh^{m-1}(r)\,dr\,d\omega,
\]
and \(p\alpha=m-1\) gives
\[
  \ee^{-p\alpha r}\sinh^{m-1}(r)
  =
  2^{-(m-1)}(1-\ee^{-2r})^{m-1}
  =
  2^{-(m-1)}+O(\ee^{-2r}).
\]
Hence, with
\(C_m=\vol(\mathbb S^{m-1})2^{-(m-1)}\),
\[
  \int_{\bH^m(-1)}u_R^p\,dV=C_mR+O(1).
\]
On \(2\leq r\leq R\), one has
\(|\nabla u_R|=\alpha u_R\).  On the transition regions,
\[
  |\nabla u_R|
  \leq(\|\chi_R'\|_\infty+\alpha)\ee^{-\alpha r},
\]
so their contribution to the \(p\)-energy is \(O(1)\), uniformly in
\(R\).  Therefore
\[
  \int_{\bH^m(-1)}|\nabla u_R|^p\,dV
  =\alpha^pC_mR+O(1).
\]
The Rayleigh quotients of \(u_R\) consequently tend to \(\alpha^p\), and
\[
  \lambda_{1,p}\bigl(\bH^m(-1)\bigr)\leq\alpha^p.
\]
Together with the lower bound, this proves
\[
  \lambda_{1,p}\bigl(\bH^m(-1)\bigr)
  =\left(\frac{m-1}{p}\right)^p
  \quad(1<p<\infty),
\]
and completes the proof.
\end{proof}

\begin{remark}\label{rem:p-spectrum-busemann}
A pointwise lower bound for \(\Delta B\) directly yields a lower bound
for the bottom of the \(p\)-spectrum.  More precisely, suppose that
\(X\) carries a function \(B\in C^2(X)\) such that \(|\nabla B|=1\)
and \(\Delta B\geq h>0\).  Applying
Lemma~\ref{lem:p-mckean-defect}, with \(m-1\) replaced by \(h\), gives
\[
  \lambda_{1,p}(X)\geq\left(\frac hp\right)^p.
\]
For the Busemann functions in Theorem~\ref{thm:p-rigidity},
Hessian comparison gives \(\Delta B_\xi\geq m-1\), and hence recovers
the McKean lower bound.
\end{remark}

\end{document}